\documentclass[11pt]{amsart}
\usepackage[a4paper, total={6in, 8in}]{geometry}
\usepackage{enumerate}
\usepackage{amsmath}
\usepackage{amssymb,latexsym}
\usepackage{amsthm}
\usepackage{color}
\usepackage{cancel}
\usepackage{graphicx}
\usepackage{cite}
\usepackage{changes}
\usepackage{hyperref}
\usepackage{comment}
\usepackage{amscd}
\usepackage{mathtools}

\usepackage[foot]{amsaddr}

\newtheorem{theorem}{Theorem}[section]

\newtheorem{lemma}[theorem]{Lemma}
\newtheorem{corollary}[theorem]{Corollary}

\newtheorem{proposition}[theorem]{Proposition}

\newtheorem{example}[theorem]{Example}

\newtheorem{remark}[theorem]{Remark}

\newcommand{\fqn}{\mathbb{F}_{q^n}}

\newcommand{\F}{{\mathbb F}}

\newcommand{\fq}{{\mathbb F}_{q}}

\newcommand{\la}{\langle}
\newcommand{\ra}{\rangle}

\newcommand{\PG}{\mathrm{PG}}
\newcommand{\N}{\mathrm{N}}

\title{Linear sets on the projective line with complementary weights}
\date{}

\author[V. Napolitano]{Vito Napolitano}
\address{Vito Napolitano, \textnormal{Dipartimento di Matematica e Fisica, Universit\`a degli Studi della Campania ``Luigi Vanvitelli'', Viale Lincoln, 5, I--\,81100 Caserta, Italy}}
\email{vito.napolitano@unicampania.it}

\author[O. Polverino]{Olga Polverino}
\address{Olga Polverino, \textnormal{Dipartimento di Matematica e Fisica, Universit\`a degli Studi della Campania ``Luigi Vanvitelli'', Viale Lincoln, 5, I--\,81100 Caserta, Italy}}
\email{olga.polverino@unicampania.it}

\author[P. Santonastaso]{Paolo Santonastaso}
\address{Paolo Santonastaso, \textnormal{Dipartimento di Matematica e Fisica, Universit\`a degli Studi della Campania ``Luigi Vanvitelli'', Viale Lincoln, 5, I--\,81100 Caserta, Italy}}
\email{paolo.santonastaso@unicampania.it}

\author[F. Zullo]{Ferdinando Zullo}
\address{Ferdinando Zullo, \textnormal{Dipartimento di Matematica e Fisica, Universit\`a degli Studi della Campania ``Luigi Vanvitelli'', Viale Lincoln, 5, I--\,81100 Caserta, Italy}}
\email{ferdinando.zullo@unicampania.it}

\begin{document}

\maketitle

\begin{abstract}
Linear sets on the projective line have attracted a lot of attention because of their link with blocking sets, KM-arcs and rank-metric codes. 
In this paper, we study linear sets having two points of complementary weight, that is with two points for which the sum of their weights equals the rank of the linear set.
As a special case, we study those linear sets having exactly two points of weight greater than one, by showing new examples and studying their equivalence issue.
Also we determine some linearized polynomials defining the linear sets recently introduced by Jena and Van de Voorde (2021).
\end{abstract}

\section{Introduction}

Let $q$ be a prime power, $\fq$ be the finite field of order $q$, and $\fqn$ be the degree $n$ extension of $\fq$.
$\fq$-linear sets in the projective space $\PG(r-1,q^n)$ are a natural generalization of subgeometries and, in recent years, they have been intensively used to construct, to classify and/or to characterize many different geometrical and algebraic objects like blocking sets, two-intersection sets, complete caps, translation spreads of the Cayley Generalized Hexagon, translation ovoids of polar spaces, semifield flocks, finite semifields and linear codes; see \cite{LavVdV,Polverino} and the references therein.

In this paper, we deal with linear sets on the projective line.
More precisely, let $V$ be a $2$-dimensional $\fqn$-vector space and let $\Lambda=\PG(V,\fqn)=\PG(1,q^n)$. 
Let $U$ be an $\fq$-subspace of $V$, then the set
\[ L_U=\{\langle \mathbf{u} \rangle_{\fqn} \colon \mathbf{u}\in U\setminus\{\mathbf{0}\}\} \]
is said to be an $\fq$-\emph{linear set} of rank $\dim_{\fq}(U)$.
The \emph{weight} of a point $P=\langle \mathbf{v}\rangle_{\fqn} \in \Lambda$ in $L_U$ is defined as $\dim_{\fq}(U \cap \langle \mathbf{v}\rangle_{\fqn})$.
One of the most difficult task when studying linear sets is to establish whether or not two linear sets are $\mathrm{P\Gamma L}(2,q^n)$-equivalent.
However, there are some necessary conditions which can help to solve this problem when $\dim_{\fq}(U)=n$.
Indeed, in such a case there exists a $q$-polynomial $f(x) \in \fqn[x]$ such that $L_U$ is mapped by a collineation of $\mathrm{PGL}(2,q^n)$ into
\[ L_f=\{\langle (x,f(x)) \rangle_{\fqn} \colon x \in \fqn^*\}. \]
The aforementioned necessary conditions involve the coefficients of $f(x)$, see e.g.\ \cite[Lemma 3.6]{CsMP}.
Nevertheless, it could be hard to find polynomials describing a fixed linear set.

The main objects of this paper are linear sets whose rank is less than or equal to $n$ and which have two points of complementary weights, that is the sum of the weights of the points equals the rank of the linear set. These linear sets turn out to be described by a projection map of $\fqn$ and this description makes possible to express them by means of a linearized polynomial when the rank of the linear set is $n$, see Section \ref{sec:projection}.
An easy example is given by the linear set defined by the trace function from $\fqn$ to $\fq$, in which there is one point of weight $n-1$ and all the other points have weight one.
This latter example belongs to the family of $i$-clubs, that is $\fq$-linear sets with one point of weight $i$ and all the others of weight one; \cite{DeBoeckVdV2016}.
So, it is natural to ask whether or not $\fq$-linear sets with exactly two points of weight greater than one exist.
As a special case of the study of linear sets with two points of complementary weights, in Section \ref{sec:2weight} we consider those having exactly two points of weight greater than one.
More precisely, we give some restriction on the parameters for their existence and we construct explicitly several examples of these linear sets, investigating the problem of the $\mathrm{\Gamma L}(2,q^n)$-equivalence of the $\fq$-subspaces representing such linear sets and we also determine some $q$-polynomials describing them.

Furthermore, very recently in \cite{DVdV} Jena and Van de Voorde introduced a family of linear sets that has the minimum number of points, generalizing the one introduced in \cite{LunPol2000}.
In \cite[Section 2.5]{DVdV}, they leave the problem of determining $q$-polynomials defining such linear sets. In Theorem \ref{th:polminsize} and Corollary \ref{cor:polminsizespecial} of Section \ref{sec:minsize} we solve this problem.
These polynomials are examples of linearized polynomials such that the set
\[ \mathrm{Im}\left( \frac{f(x)}x\right)=\left\{ \frac{f(x)}x \colon x \in \fqn^* \right\} \]
has the minimum size among linearized polynomials over $\fqn$ with maximum field of linearity $\F_q$, see \cite{Ball,BBBSSz}.
Up to our knowledge, the only previously known polynomial satisfying this property was the trace function.


We conclude the paper by showing in Appendix \ref{sec:pg1q4} the list of the $q$-polynomials defining $\fq$-linear sets in $\PG(1,q^4)$, which was an open problem stated in \cite{CsMP} and in \cite{BMZZ}.

\section{Preliminaries}

In this section we briefly recall some definitions and properties of linearized polynomials, linear sets and dual basis which will play a special role in our results.

\subsection{Linearized polynomials}

Let $q=p^h$, where $h$ is a positive integer and $p$ is a prime.
A $\sigma$-\emph{linearized polynomial} (or a $\sigma$-\emph{polynomial}, or $q$-\emph{polynomial} if $\sigma $ is the map $x\in \fqn\mapsto x^q\in \fqn$) over $\F_{q^n}$ is a polynomial of the form
\[f(x)=\sum_{i=0}^t a_i x^{\sigma^i},\]
where $a_i\in \F_{q^n}$, $t$ is a non-negative  integer and $\sigma$ is a generator of the Galois group $\mathrm{Gal}(\F_{q^n}\colon \F_q)$.
Furthermore, if $a_t \neq 0$ we say that $t$ is the $\sigma$-\emph{degree} of $f(x)$.
We will denote by $\mathcal{L}_{n,\sigma}$ the set of all $\sigma$-polynomials over $\F_{q^n}$ (or simply by $\mathcal{L}_{n,q}$ if $\sigma\colon x \in \F_{q^n}\mapsto x^q\in \F_{q^n}$).

\cite[Theorem 5]{GQ2009} and \cite[Theorem 10]{GQ2009x} imply the following bound on the number of roots and the following necessary condition to have the maximum number of roots over $\fqn$ of a $\sigma$-polynomial. 

\begin{theorem}\label{Gow}
Consider
\[f(x)=a_0x+a_1x^{\sigma}+\cdots+a_{k-1}x^{\sigma^{k-1}}+a_kx^{\sigma^k}\in \mathcal{L}_{n,\sigma},\]
with $k\leq n-1$ and let $a_0,a_1,\ldots,a_k$ be elements of $\F_{q^n}$ not all of them zero. Then 
\[ \dim_{\F_q}(\ker (f(x)))\leq k. \]
Moreover, if $\dim_{\F_q}(\ker (f(x)))=k$ then $\N_{q^n/q}(a_0)=(-1)^{nk}\N_{q^n/q}(a_k)$.
\end{theorem}

Recall that the trace function from $\fqn$ over $\fq$ is defined by:
\[
\mathrm{Tr}_{q^n/q}: a \in \F_{q^n} \mapsto \sum_{i=0}^{n-1} a^{q^i} \in \F_q,
\]
and it is a linear map that defines a nondegenerate symmetric bilinear form as follows:
\[
(a,b) \in \F_{q^n} \times \F_{q^n} \mapsto \mathrm{Tr}_{q^n/q}(ab) \in \F_q.
\]
The \emph{adjoint} of a $\sigma$-polynomial $f(x)=f_0x+f_1x^{\sigma}+\ldots+f_{n-1}x^{\sigma^{n-1}}$ with respect to the trace bilinear form is 
$$f^\top(x)=\sum_{i=0}^{n-1}{\sigma^{n-i}}(f_i)x^{\sigma^{n-i}}.$$
Such a $\sigma$-polynomial satisfies 
$$\mathrm{Tr}_{q^n/q}(f(\alpha)\beta)=\mathrm{Tr}_{q^n/q}(\alpha{f}^\top(\beta)), \qquad \mbox{ for every } \alpha,\beta \in \fqn.$$

Sheekey in \cite{John} gave the following definition.  Let $f(x) \in \mathcal{L}_{n,q}$. Then $f$ is said to be \emph{scattered} if for every $x,y \in \fqn^*$
\[ \frac{f(x)}{x}=\frac{f(y)}{y}\,\, \Leftrightarrow\,\, \frac{y}x\in \fq, \]
or equivalently
\[ \dim_{\fq}(\ker(f(x)-mx))\leq 1, \,\,\,\text{for every}\,\,\, m \in \fqn. \]
As we will see later, the term \emph{scattered} arises form a geometric framework; see \cite{BL2000}. Scattered polynomials have attracted a lot of attention because of their connections with maximum distance rank metric codes; see e.g. \cite{BM,BZ,BZZ2021,FM}.

\subsection{Linear sets}

Let $V$ be a $r$-dimensional vector space over $\fqn$ and let $\Lambda=\PG(V,\F_{q^n})=\PG(r-1,q^n)$.
Let $U$ be an $\fq$-subspace of $V$ of dimension $k$, then the set of points
\[ L_U=\{\la {\bf u} \ra_{\mathbb{F}_{q^n}} : {\bf u}\in U\setminus \{{\bf 0} \}\}\subseteq \Lambda \]
is said to be an $\fq$-\emph{linear set of rank $k$}.
Let $P=\langle \mathbf{v}\rangle_{\fqn}$ be a point in $\Lambda$. The \emph{weight of $P$ in $L_U$} is defined as 
\[ w_{L_U}(P)=\dim_{\fq}(U\cap \langle \mathbf{v}\rangle_{\fqn}). \]
If $N_i$ denotes the number of points of $\Lambda$ having weight $i\in \{0,\ldots,k\}$  in $L_U$, the following relations hold:
\begin{equation}\label{eq:card}
    |L_U| \leq \frac{q^k-1}{q-1},
\end{equation}
\begin{equation}\label{eq:pesicard}
    |L_U| =N_1+\ldots+N_k,
\end{equation}
\begin{equation}\label{eq:pesivett}
    N_1+N_2(q+1)+\ldots+N_k(q^{k-1}+\ldots+q+1)=q^{k-1}+\ldots+q+1.
\end{equation}

It is easy to see that the weight distribution of two linear sets defined by two $\Gamma\mathrm{L}(r,q^n)$-equivalent $\fq$-subspaces is the same.

\begin{proposition}\label{prop:weightinvariace}
Let $U_1$ and $U_2$ be two $k$-dimensional $\Gamma\mathrm{L}(r,q^n)$-equivalent $\fq$-subspaces in $V$. For every $i \in \{0,1,\ldots,k\}$, the number of points of weight $i$ in $L_{U_1}$ coincides with the number of points of weight $i$ in $L_{U_2}$.
\end{proposition}

Furthermore, $L_U$ is called \emph{scattered} if it has the maximum number $\frac{q^k-1}{q-1}$ of points, or equivalently, if all points of $L_U$ have weight one.
Blokhuis and Lavrauw in \cite{BL2000} proved that the rank of a scattered linear set is bounded by $\frac{rn}2$ and, after a series of papers \cite{BBL2000,BGMP2015,CsMPZu,CSMPZ2016}, this bound is proven to be tight when $rn$ is even.

\begin{theorem}\cite{BBL2000,BGMP2015,CsMPZu,CSMPZ2016}\label{th:existscattered}
For each positive integers $r$ and $n$ such that $rn$ is even, for each value of $q$, there exists a scattered $\fq$-linear set in $\mathrm{PG}(r-1,q^n)$ of rank $\frac{rn}2$. 
\end{theorem}

When $rn$ is odd less constructions are known, see $L_{W}$ in Example \ref{ex:scattered} and \cite[Section 5]{BCsMT}.

\begin{example}\label{ex:scattered}
Let $V=\F_{q^t}^\ell$. Suppose that $\ell$ is odd and $\ell=2h +1$, then consider
\[ W=\{ (x_1,x_1^{q^{s_1}},x_2,x_2^{q^{s_2}},\ldots,x_h,x_h^{q^{s_h}},\alpha) \colon x_i \in \F_{q^t}, \alpha \in \fq \}, \]
with $\gcd(s_1,\ldots,s_h,t)=1$.
Again, $L_{W}$ is a scattered $\fq$-linear set in $\PG(\ell-1,q^t)$ of rank $\frac{(\ell-1) t}2+1$.
\end{example}

We also underline that if $L_{U}$ is a scattered $\fq$-linear set and $W$ is an $\fq$-subspace of $U$, then also $L_W$ is scattered.

If two $\fq$-subspaces $U_1$ and $U_2$ of $V$ belong to the same orbit of $\Gamma\mathrm{L}(r,q^n)$, then $L_{U_1}$ and $L_{U_2}$ are clearly $\mathrm{P}\Gamma \mathrm{L}(r,q^n)$-equivalent.
This sufficient condition is not necessary in general: there exist $\fq$-linear sets $L_{U_1}$ and $L_{U_2}$ in the same orbit of ${\rm P\Gamma L}(r,q^n)$ such that $U_1$ and $U_2$ are not in the same ${\rm \Gamma L}(r,q^n)$-orbit.
The $\fq$-linear set $L_U$ is called \emph{simple} if every $\fq$-subspace $W$ of $V$ satisfying $L_U=L_W$ is in the same $\Gamma\mathrm{L}(r,q^n)$-orbit as $U$. 
It is easy to see that simple linear sets $L_{U_1}$ and $L_{U_2}$ are in the same $\mathrm{P\Gamma L}(r,q^n)$-orbit if and only if $U_1$ and $U_2$ are $\mathrm{\Gamma L}(r,q^n)$-equivalent.
For further details on the equivalence issue for linear sets we refer to \cite{CsMP}.

Some of the linear sets considered in this paper are $\fq$-linear sets $L$ of rank $n$ in $\PG(1,q^n)$.
Since ${\rm P\Gamma L}(1,q^n)$ is $3$-transitive on $\PG(1,q^n)$, we can assume up to equivalence that $L$ does not contain the point $\langle(0,1)\rangle_{\fqn}$.
This implies that $L$ equals $L_f=L_{U_f}$, where
\[U_{f}=\{(x,f(x)) \colon x \in \fqn\}\] 
for some $q$-polynomial $f(x)\in\mathcal{L}_{n,q}$.
Note that the polynomial $f(x)$ is scattered if and only if $L_f$ is a scattered $\fq$-linear set.

Now, we briefly recall the notion of the dual of a linear set.
Let $\sigma \colon V \times V \rightarrow \fqn$ be a nondegenerate reflexive sesquilinear form on the $\fqn$-vector space $V$ and consider $\sigma'=\mathrm{Tr}_{q^n/q} \circ \sigma$, which is a nondegenerate reflexive sesquilinear form on $V$ seen as an $\fq$-vector space of dimension $2n$. Then we may consider $\perp$ and $\perp'$ as the orthogonal complement maps defined by $\sigma$ and $\sigma'$, respectively, and $\tau$ and $\tau'$ as the polarities of $\PG(V,\fqn)$ and $\PG(V,\fq)$, respectively. Let $L_U$ be an $\fq$-linear set in $\PG(V,\fqn)$ of rank $t$ then $L_U^\tau=L_{U^{\perp'}}$ is an $\fq$-linear set in $\PG(V,\fqn)$ of rank $2n-t$ that is called the \emph{dual linear set of $L_U$} with respect to the polarity $\tau$.
As proved in \cite[Proposition 2.5]{Polverino}, if $\tau_1$ and $\tau_2$ are two polarities as above, that the dual linear sets $L_U^{\tau_1}$ and $L_U^{\tau_2}$ are $\mathrm{P\Gamma L}(2,q^n)$-equivalent.
Therefore the definition of the dual linear set does not depend on the choice of the polarity.
So, we may choose $V=\F_{q^{n}}\times \fqn$ and $\sigma \colon V\times V \rightarrow \fqn$ such that
\[ \sigma((x,y),(u,v))=xv-yu. \]
Let $\sigma'=\mathrm{Tr}_{q^n/q} \circ \sigma$ and denote by $\tau$ the polarity associated with $\sigma$.
If $L_f$ is an $\fq$-linear set in $\PG(1,q^n)$ of rank $n$ for some $f(x) \in \mathcal{L}_{n,q}$, then $L_f^\tau$ the dual of $L_f$ is $L_{f^\top}$; see e.g. \cite[Lemma 2.6]{BGMP2015}.

We refer to \cite{LavVdV} and \cite{Polverino} for comprehensive references on linear sets.

\subsection{Dual basis and the projection map}

Two ordered $\F_{q}$-bases $\mathcal{B}=(\xi_0,\ldots,\xi_{n-1})$ and $\mathcal{B}^*=(\xi_0^*,\ldots,\xi_{n-1}^*)$ of $\F_{q^n}$ are said to be \emph{dual bases} if $\mathrm{Tr}_{q^n/q}(\xi_i \xi_j^*)= \delta_{ij}$, for $i,j\in\{0,\ldots,n-1\}$, where $\delta_{ij}$ is the Kronecker symbol. It is well known that for any $\F_q$-basis $\mathcal{B}=(\xi_0,\ldots,\xi_{n-1})$ there exists a unique dual basis $\mathcal{B}^*=(\xi_0^*,\ldots,\xi_{n-1}^*)$ of $\mathcal{B}$, see e.g.\ \cite[Definition 2.30]{lidl_finite_1997}. 

In next result it is shown a procedure that can be performed to determine the dual basis of a fixed one.

\begin{proposition}\cite[Remark 4.10]{wl} \label{prop:dualbasesgeneral}
Let $\mathcal{B}=(\xi_0,\ldots,\xi_{n-1})$ be an $\F_q$-basis of $\F_{q^n}$. Let 
\[
V=(\xi_j^{q^i})=
\left(
\begin{matrix}
\xi_0 & \xi_1 & \cdots & \xi_{n-1} \\
\xi_0^q & \xi_1^q & \cdots & \xi_{n-1}^q \\
\vdots & \vdots & \ddots & \vdots \\
\xi_0^{q^{n-1}} & \xi_1^{q^{n-1}} & \cdots & \xi_{n-1}^{q^{n-1}} \\
\end{matrix}
\right).
\]
Let $\tilde \xi_i$ be the $(0,i)$-cofactor of the matrix $V$, for $i\in\{0,\ldots,n-1\}$. Let $\mathcal{B}^*=( \xi_0^*,\ldots,\xi_{n-1}^*) $ be the dual basis of $\mathcal{B}$. Then 
\[
\xi_i^*= \frac{\tilde \xi_i}{\det(V)}, \ \ \ \mbox{for }i\in\{0,\ldots,n-1\}.
\]
\end{proposition}

When considering polynomial basis we have the following results.

\begin{lemma}\cite[Exercise 2.40]{lidl_finite_1997}\label{lem:dualbasispol}
Let $\lambda \in \fqn$ such that $\mathcal{B}=(1,\lambda,\ldots,\lambda^{n-1})$ is an ordered $\fq$-basis.
Let $f(x)=a_0+a_1x+\ldots+a_{n-1}x^{n-1}+x^n$ be the minimal polynomial of $\lambda$ over $\fq$. Let 
\[
\frac{f(x)}{x-\lambda}=b_0+b_1x+\cdots+b_{n-1}x^{n-1} \mbox{ and } \delta=f'(\lambda).
\]
Then the dual basis $\mathcal{B}^*$ of $\mathcal{B}$ is 
\[ \mathcal{B}^*=(\delta^{-1}b_0,\delta^{-1}b_1,\ldots,\delta^{-1}b_{n-1}). \]
\end{lemma}

\begin{corollary}\label{cor:dualbasis}
Let $\lambda \in \fqn$ such that $\mathcal{B}=(1,\lambda,\ldots,\lambda^{n-1})$ is an ordered $\fq$-basis of $\fqn$.
Let $f(x)=a_0+a_1x+\ldots+a_{n-1}x^{n-1}+x^n$ be the minimal polynomial of $\lambda$ over $\fq$. 
Then the dual basis $\mathcal{B}^*$ of $\mathcal{B}$ is 
\[ \mathcal{B}^*=(\delta^{-1}\gamma_0,\ldots,\delta^{-1}\gamma_{n-1}), \]
where $\delta=f'(\lambda)$ and $\gamma_i=\sum_{j=1}^{n-i} \lambda^{j-1}a_{i+j}$, for every $i \in \{0,\ldots,n-1\}$. 
\end{corollary}

\begin{proof}
Let 
\[
\frac{f(x)}{x-\lambda}=b_0+b_1x+\cdots+b_{n-1}x^{n-1}.
\]
From Lemma \ref{lem:dualbasispol}, we have that the dual basis $\mathcal{B}^*$ of $\mathcal{B}$ is 
\[ \mathcal{B}^*=(\delta^{-1}b_0,\delta^{-1}b_1,\ldots,\delta^{-1}b_{n-1}). \]
Now we explicit the $b_i$'s. From the identity $f(x)=(x-\lambda)(b_0+b_1x+\cdots+b_{n-1}x^{n-1})$, we get
\[
b_{n-i}=\lambda^{i-1}+\lambda^{i-2}a_{n-1}+\lambda^{i-3}a_{n-2}+\cdots+ \lambda a_{n-i+2}+ a_{n-i+1},
\]
that is
\[
b_h=\sum_{j=1}^{n-h} \lambda^{j-1}a_{h+j}.
\]
\end{proof}

\begin{remark}
In the next corollary, we analyze the case when the polynomial basis is \emph{weakly self dual}.
Let $\mathcal{B}=(b_0,\ldots,b_{n-1})$ an ordered $\fq$-basis of $\fqn$ and let $\mathcal{B}^*=(b_0^*,\ldots,b_{n-1}^*)$ its dual basis. We say that $\mathcal{B}$ is \emph{weakly self dual} if there exists $\delta \in \fqn^*$ and a permutation $\pi \in S_n$ such that $b_i^*=\delta b_{\pi(i)}$, for every $i \in \{0,\ldots,n-1\}$.
In \cite{Gollmann} and in \cite{Blake}, it has been proved that $\mathcal{B}=(1,\lambda,\ldots,\lambda^{n-1})$ is a weakly self dual basis if and only if the minimal polynomial of $\lambda$ over $\fq$ is either a binomial $f(x)=x^n-d \in \fq[x]$ or a trinomial $f(x)=x^n-cx^k-1 \in \fq[x]$. See also \cite{Cao}.
\end{remark}

\begin{corollary}\label{cor:binetrin}
Let $\lambda \in \fqn$ such that $\mathcal{B}=(1,\lambda,\ldots,\lambda^{n-1})$ is an ordered $\fq$-basis of $\fqn$.
\begin{itemize}
    \item If $f(x)=x^n-d$ is the minimal polynomial of $\lambda$ over $\fq$, then the dual basis of $\mathcal{B}$ is
    \[ \mathcal{B}^*=\left(\frac{\lambda^n}{nd},\frac{\lambda^{n-1}}{nd},\ldots,\frac{\lambda}{nd},\frac{1}{nd}\right). \]
    \item If $f(x)=x^n-cx^k-1$ is the minimal polynomial of $\lambda$ over $\fq$, then the dual basis of $\mathcal{B}$ is
    \[ \mathcal{B}^*=(\delta^{-1}\lambda^{k-1},\delta^{-1}\lambda^{k-2},\ldots,\delta^{-1},\delta^{-1}\lambda^{n-1},\ldots,\delta^{-1}\lambda^k), \]
    where $\delta=\frac{n\lambda^{n-1}-ck\lambda^{k-1}}{-c+\lambda^{n-k}}$.
\end{itemize}
\end{corollary}
\begin{proof}
First suppose that $f(x)=x^n-d$. By Corollary \ref{cor:dualbasis}, the dual basis of $\mathcal{B}$ is given by
\[ (\delta^{-1}\gamma_0,\ldots,\delta^{-1}\gamma_{n-1}), \]
where $\delta=n\lambda^{n-1}$ and $\gamma_i=\lambda^{n-i-1}$.
Now, assume that $f(x)=x^n-cx^k-1$. Again, applying  Corollary \ref{cor:dualbasis}, the dual basis of $\mathcal{B}$ is given by
\[ (\overline{\delta}^{-1}\gamma_0,\ldots,\overline{\delta}^{-1}\gamma_{n-1}), \]
where $\overline{\delta}=n\lambda^{n-1}-ck\lambda^{k-1}$ and $\gamma_i=\lambda^{k-i-1}(-c+\lambda^{n-k})$ if $i<k$, and $\gamma_i=\lambda^{n+k-i-1}\lambda^{-k}=\lambda^{n+k-i-1}(-c+\lambda^{n-k})$ if $i\geq k$.
The assertion then follows since $\delta=\frac{n\lambda^{n-1}-ck\lambda^{k-1}}{-c+\lambda^{n-k}}$.
\end{proof}

In the following lemma, we determine polynomial forms of the projection maps of $\fqn$ by means of dual basis.

\begin{lemma}\label{lem:proiez}
Let $S$ and $T$ be $\fq$-subspaces of $\fqn$ such that $\fqn=S \oplus T$ and let $p_{T,S}$ be the $\fq$-linear map of $\fqn$ projecting an element of $\fqn$ from $T$ onto $S$.
Then 
\[ p_{T,S}(x)=\sum_{i=t}^{n-1}\xi_i\mathrm{Tr}_{q^n/q}(\xi_i^*x)=\sum_{j=0}^{n-1}\left( \sum_{i=t}^{n-1} \xi_i\cdot \xi_i^{*q^j} \right)x^{q^j}, \]
where $\mathcal{B}_T=\{\xi_0,\ldots,\xi_{t-1}\}$ is an $\fq$-basis of $T$, $\mathcal{B}_S=\{\xi_t,\ldots,\xi_{n-1}\}$ is an $\fq$-basis of $S$ and $\mathcal{B}^*=(\xi_0^*,\ldots,\xi_{n-1}^*)$ is the dual basis of $\mathcal{B}=(\xi_0,\ldots,\xi_{n-1})$. 
\end{lemma}
\begin{proof}
Every $x \in \fqn$ can be uniquely written as $x=u+v$ where $u\in S$ and $v \in T$, and so $p_{T,S}(x)=u$.
Denote by $\overline{p}(x)=\sum_{i=t}^{n-1}\xi_i\mathrm{Tr}_{q^n/q}(\xi_i^*x)$.
We show that $p_{T,S}(x)=\overline{p}(x)$ for every $x \in \fqn$.
First, let $j \in \{0,\ldots,t-1\}$ then
\[ \overline{p}(\xi_j)=\sum_{i=t}^{n-1}\xi_i\mathrm{Tr}_{q^n/q}(\xi_i^*\xi_j)=0, \]
since $\mathrm{Tr}_{q^n/q}(\xi_i^*\xi_j)=0$ for every $i\geq t$.
So, $\ker(\overline{p})\supseteq T$.
Now, let $j\geq t$ then
\[ \overline{p}(\xi_j)=\sum_{i=t}^{n-1}\xi_i\mathrm{Tr}_{q^n/q}(\xi_i^*\xi_j)=\xi_j, \]
so that $\overline{p}(x)$ is the projection from $T$ onto $S$. 
\end{proof}

\section{Linear sets described by projection maps}\label{sec:projection}

In this section we introduce an efficient machinery to determine $\fq$-linearized polynomials describing a linear set of rank $n$ in the projective line $\PG(1,q^n)$ of a certain form.

The family we consider can be obtained as follows.
Let $S$ and $T$ be $\fq$-subspaces of $\fqn$ such that $\dim_{\fq}(S)+\dim_{\fq}(T)=k\leq n$.
Let $t=\dim_{\fq}(T)$, $s=\dim_{\fq}(S)$ and 
\[ U=T \times S\subseteq \fqn^2. \]
Hence
\begin{equation}\label{eq:formlinearset2}
    L_U=\{\langle (a,b) \rangle_{\fqn} \colon a \in T, b \in S, (a,b)\ne(0,0)\}.
\end{equation}
Clearly, $\dim_{\fq}(U)=k$, $w_{L_U}(\langle (1,0)\rangle_{\fqn})=t$ and $w_{L_U}(\langle (0,1)\rangle_{\fqn})=s$.
For those linear sets $L_U$ constructed as above and having at least one point of weight one in $L_U$, we have the following bounds on the number of its points.

\begin{proposition}
Let $L_U$ be an $\fq$-linear set as in \eqref{eq:formlinearset2}. If at least one point of weight one in $L_U$ then
\begin{equation}\label{eq:bound}
  q^{k-1}+1  \leq |L_U|\leq q^{k-1}+\ldots+q^{\max\{t,s\}}-q^{\min\{t,s\}-1}-\ldots-q+1.
\end{equation}
\end{proposition}
\begin{proof}
The lower bound follows by \cite[Theorem 1.2]{DeBeuleVdV} (and for $k=n$ by \cite[Lemma 2.2]{BoPol}) and the upper bound is a consequence of \eqref{eq:card}, \eqref{eq:pesicard} and \eqref{eq:pesivett}.
\end{proof}

Up to the action of $\mathrm{PGL}(2,q^n)$ on $\PG(1,q^n)$, linear sets of the shape \eqref{eq:formlinearset2} describe all linear sets of rank $k$ admitting two points of complementary weights, as pointed out in the next result.
To this aim we need the following notation.
If $S$ and $T$ are two subsets of $\fqn$, we denote by $S\cdot T$ the set $\{ s\cdot t \colon s \in S, t \in T\}$ and by $S^{-1}$ the set $\{s^{-1} \colon s \in S\setminus\{0\}\}$.

\begin{proposition}\label{th:UasUST}
If $L_W$ is an $\fq$-linear set of rank $k$ in $\PG(1,q^n)$ for which there exist two distinct points $P,Q \in L_W$ such that
$w_{L_W}(P)=s$, $w_{L_W}(Q)=t$ and $s+t=k\leq n$. Then, 
$L_W$ is $\mathrm{PG}\mathrm{L}(2,q^n)$-equivalent to a linear set  $L_U$ of shape \eqref{eq:formlinearset2}, where $U=T \times S$,
for some $\fq$-subspaces $S$ and $T$ of $\fqn$ with $\dim_{\fq}(S)=s$, $\dim_{\fq}(T)=t$. Also, we can assume $T\cap S=\{0\}$.
\end{proposition}
\begin{proof}
Let $\varphi \in \mathrm{PGL}(2,q^n)$ be a collineation of $\PG(1,q^n)$ such that $\varphi(Q)=\langle (1,0)\rangle_{\fqn}$ and $\varphi(P)=\langle (0,1)\rangle_{\fqn}$, and let $f$ be the associated linear isomorphism of $\fqn^2$. 
Then $\varphi(L_W)=L_{U'}$ where $U'=f(W)$, $w_{L_{U'}}(\langle (1,0)\rangle_{\fqn})=t$ and $w_{L_{U'}}(\langle (0,1)\rangle_{\fqn})=s$.
Hence,
\[ U'=(U'\cap \langle (1,0)\rangle_{\fqn})\oplus (U'\cap \langle (0,1)\rangle_{\fqn}). \]
Choosing $T'$ and $S'$ as the $\fq$-subspaces obtained in the following way $T'\times \{0\}=U'\cap \langle (1,0)\rangle_{\fqn}$ and $\{0\}\times S'=U'\cap \langle (0,1)\rangle_{\fqn}$, then $L_W$ is $\mathrm{PG}\mathrm{L}(2,q^n)$-equivalent to a linear set $L_{U'}$ of shape \eqref{eq:formlinearset2}, with $U'=T'\times S'$.
Now, we prove the existence of $\lambda \in \fqn^*$ such that $T'\cap \lambda S'=\{0\}$.
By contradiction, suppose that for every $\lambda \in \fqn^*$ there exist $x \in T'\setminus \{0\}$ and $y \in S'\setminus\{0\}$ such that $x=\lambda y$.
It follows that
\begin{equation}\label{eq:TS-1} 
T'\cdot S'^{-1}=\{t s^{-1} \colon s \in S'\setminus\{0\}, t \in T'\}=\fqn.
\end{equation}
Since $\langle (x/y,1) \rangle_{\fqn} \in L_{U'}$ for every $x \in T'$ and $y \in S'\setminus\{0\}$, by \eqref{eq:TS-1} we get $\langle (\eta,1) \rangle_{\fqn} \in L_{U'}$ for every $\eta \in \fqn$.
So, it follows that $|L_{U'}|= q^n+1$, which is a contradiction since the rank of $L_{U'}$ is $\dim_{\fq}(S')+\dim_{\fq}(T')\leq n$.
Therefore, choosing $\lambda \in \fqn^*$ such that $T'\cap \lambda S'=\{0\}$ and setting $T=T'$ and $S=\lambda S'$, the assertion follows.
\end{proof}

The following is a useful criteria to study the $\Gamma\mathrm{L}(2,q^n)$-equivalence of $\fq$-subspaces involved in this family of linear sets.

\begin{lemma}\label{lem:criteria}
Consider $S$, $S'$, $T$ and $T'$ $\fq$-subspaces of $\fqn$ such that $\dim_{\fq}(S)+\dim_{\fq}(T)=k\leq n$ and $\dim_{\fq}(S')+\dim_{\fq}(T')=k\leq n$.
Let 
\[ U=T \times S, \]
and 
\[ U'=T' \times S'. \]
Suppose that $\langle(1,0)\rangle_{\fqn}$ and $\langle(0,1)\rangle_{\fqn}$ are the only points in $L_U$ and $L_{U'}$ having weight $t$ and $s$, respectively. 
\begin{itemize}
    \item If $s= t$, then $U$ and $U'$ are $\Gamma\mathrm{L}(2,q^n)$-equivalent if and only if 
    \[ \text{either}\,\, S'=\lambda S^\rho\,\, \text{and}\,\, T'=\mu T^\rho,\,\,\text{or}\,\, S'=\lambda T^\rho\,\, \text{and}\,\, T'=\mu S^\rho, \]
    where $\lambda,\mu\in \fqn^*$ and $\rho \in \mathrm{Aut}(\fqn)$.
    \item If $s\ne t$, then $U$ and $U'$ are $\Gamma\mathrm{L}(2,q^n)$-equivalent if and only if 
    \[ S'=\lambda S^\rho\,\, \text{and}\,\, T'=\mu T^\rho, \]
    where $\lambda,\mu\in \fqn^*$ and $\rho \in \mathrm{Aut}(\fqn)$.
\end{itemize}
\end{lemma}
\begin{proof}
Suppose that $U$ and $U'$ are $\Gamma\mathrm{L}(2,q^n)$-equivalent, then there exist a matrix $\left(\begin{array}{cc} a & b\\ c & d\end{array}\right) \in \mathrm{GL}(2,q^n)$ and $\rho \in \mathrm{Aut}(\fqn)$ such that
\[ \left(\begin{array}{cc} a & b\\ c & d \end{array}\right) U^\rho =U'. \]
Since $\langle(1,0)\rangle_{\fqn}$ and $\langle(0,1)\rangle_{\fqn}$ are the only points in $L_U$ and $L_{U'}$ having weight $s$ and $t$, and the elements of $\mathrm{\Gamma L}(2,q^n)$ preserves the weight of points, we get
\begin{itemize}
    \item if $s=t$, then either $b=c=0$ or $a=d=0$;
    \item if $s\ne t$, then $b=c=0$.
\end{itemize}
So that 
\begin{itemize}
    \item if $s=t$, then either $T'=a T^\rho$ and $S'=a S^\rho$, or $T'=c S^\rho$ and $S'=b T^\rho$;
    \item if $s\ne t$, then $T'=aT^\rho$ and $S'=dS^\rho$.
\end{itemize}
\end{proof}

Now, our aim is to determine a polynomial form for linear sets of shape \eqref{eq:formlinearset2} having rank $n$.

\begin{theorem}\label{clas}
Let $L$ be an $\fq$-linear set of rank $n$ in $\PG(1,q^n)$ admitting a point $P$ of weight $t$ and a point $Q\ne P$ of weight $s$ with $t+s=n$.
Then $L$ is  $\mathrm{PGL}(2,q^n)$-equivalent to the $\fq$-linear set
\[L_{p_{T,S}}=\{\langle (x,p_{T,S}(x)) \rangle_{\fqn} \colon x \in \fqn^*\},\]
where $T$ and $S$ are $\fq$-subspaces of $\fqn$ of dimension $t$ and $s$, respectively, such that $\fqn=T\oplus S$ and $p_{T,S}(x)$ is the projection map of $\fqn$ from $T$ onto $S$.
Moreover, if $T=\langle \xi_0,\ldots,\xi_{t-1}\rangle_{\fq}$, $S=\langle \xi_t,\ldots,\xi_{n-1} \rangle_{\fq}$ and $\mathcal{B}^*=(\xi_0^*,\ldots,\xi_{n-1}^*)$ is the dual basis of $\mathcal{B}=(\xi_0,\ldots,\xi_{n-1})$, then
\[ p_{T,S}(x)=\sum_{i=t}^{n-1}\xi_i\mathrm{Tr}_{q^n/q}(\xi_i^*x)=\sum_{j=0}^{n-1}\left( \sum_{i=t}^{n-1} \xi_i\cdot \xi_i^{*q^j} \right)x^{q^j}. \]
\end{theorem}
\begin{proof}
Because of Proposition \ref{th:UasUST}, for every such $\fq$-linear set $L$ there exist two $\fq$-subspaces $S$ and $T$ of $\fqn$ such that $L$ is $\mathrm{PGL}(2,q^n)$-equivalent to $L_U$, where $U=T\times S$ and $T\cap S=\{0\}$.
Also $t=\dim_{\fq}(T)$ and $s=\dim_{\fq}(S)$, hence $t+s=n$.
Note that $\fqn=T\oplus S$.
Hence, by applying the $\fqn$-linear transformation $\phi$ of $\fqn\times\fqn$ defined by $(a,b)\mapsto \left(\left(\begin{array}{cc} 1 & 1 \\ 0 & 1 \end{array}\right)\left(\begin{array}{ll}a\\ b \end{array}\right)\right)^t$ we obtain
\[ \phi(U)=\{ (x,p_{T,S}(x)) \colon x \in \fqn \}, \]
where $p_{T,S}(x)$ is the projection map from $T$ onto $S$.
The assertion then follows from Lemma \ref{lem:proiez}.
\end{proof}

\begin{remark}
Note that by applying to $\overline{\phi}$ the $\fqn$-linear isomorphism of $\fqn^2$ to $U=T\times S$ defined by $(a,b)\mapsto \left( \left(\begin{array}{cc} 1 & 1 \\ 1 & 0 \end{array}\right) \left(\begin{array}{cc} a \\ b \end{array}\right)\right)^t$ we obtain that $L_U$ is $\mathrm{PGL}(2,q^n)$-equivalent to the $\fq$-linear set $L_{p_{S,T}}$, where 
\[ p_{S,T}(x)=\sum_{i=0}^{t-1}\xi_i\mathrm{Tr}_{q^n/q}(\xi_i^*x)=\sum_{j=0}^{n-1}\left( \sum_{i=0}^{t-1} \xi_i\cdot \xi_i^{*q^j} \right)x^{q^j}, \]
with the notation of Theorem \ref{clas}, is the projection map of $\fqn$ from $S$ onto $T$.
\end{remark}

Note that to compute $\mathcal{B}^*$ one could use Proposition \ref{prop:dualbasesgeneral}, once $\mathcal{B}$ is fixed.\\

\begin{corollary}
Let $T$ and $S$ be two $\fq$-subspaces of $\fqn$ of dimension $t$ and $s$, respectively, such that $\fqn=T\oplus S$ and let $p_{T,S}(x)$ be the projection map of $\fqn$ from $T$ onto $S$.
Suppose that $T=\langle \xi_0,\ldots,\xi_{t-1}\rangle_{\fq}$, $S=\langle \xi_t,\ldots,\xi_{n-1} \rangle_{\fq}$ and $\mathcal{B}^*=(\xi_0^*,\ldots,\xi_{n-1}^*)$ is the dual basis of $\mathcal{B}=(\xi_0,\ldots,\xi_{n-1})$. Then the dual of $L_{p_{T,S}}$ is $L_{p_{T',S'}}$, where $T'=\langle \xi_0^*,\ldots,\xi_{t-1}^*\rangle_{\fq}$ and $S'=\langle \xi_t^*,\ldots,\xi_{n-1}^* \rangle_{\fq}$.
\end{corollary}
\begin{proof}
The dual of $L_{p_{T,S}}$ is $L_{p_{T,S}^\top}$ and
\[ p_{T,S}(x)^\top=\sum_{j=0}^{n-1}\left( \sum_{i=t}^{n-1} \xi_i\cdot \xi_i^{*q^j} \right)^{q^{n-j}}x^{q^{n-j}}=\sum_{h=0}^{n-1}\left( \sum_{i=t}^{n-1} \xi_i^{q^h}\cdot \xi_i^{*} \right)x^{q^{h}}=p_{T',S'}(x), \]
by Lemma \ref{lem:proiez}.
\end{proof}

As far as we know, examples of $\fq$-linear sets of rank $n$ of the projective line $\PG(1,q^n)$ admitting exactly two points of weight greater than one are known only in the case $n=4$ and are constructed in \cite{BoPol} (see also \cite{CsZ2018}).
In Section \ref{sec:2weight} we first prove that they exist if and only if $n$ is even and then we exhibit two families of such $\fq$-linear sets of $\PG(1,q^n)$ for every $n$ even and $n\geq 4$.

\section{Linear sets with exactly two points of weight greater than one}\label{sec:2weight}

In this section we deal with $\fq$-linear sets of rank $k$ in $\PG(1,q^n)$ of shape \eqref{eq:formlinearset2} admitting exactly two points of weight greater than one.
We are able to characterize such linear sets as follows.

\begin{theorem}\label{th:SS-1}
Let $S$ and $T$ be two $\fq$-subspaces of $\fqn$ such that $\dim_{\fq}(S)+\dim_{\fq}(T)=k\leq n$.
Let $t=\dim_{\fq}(T)$, $s=\dim_{\fq}(S)$, $t\geq s$ and $U=T \times S$.
Then, the following are equivalent
\begin{itemize}
    \item[i)] $|L_U|=q^{k-1}+\ldots+q^{t}-q^{s-1}-\ldots-q+1$;
    \item[ii)] all the points of $L_U$ different from $\langle(1,0)\rangle_{\fqn}$ and $\langle(0,1)\rangle_{\fqn}$ have weight one in $L_U$;
    \item[iii)] $S\cdot S^{-1} \cap T\cdot T^{-1}=\fq$;
    \item[iv)] $\dim_{\fq}(S\cap \alpha T)\leq 1$, for every $\alpha \in \fqn^*$.
\end{itemize}
\end{theorem}
\begin{proof}
Clearly, $w_{L_U}(\langle(1,0)\rangle_{\fqn})=t$ and $w_{L_U}(\langle(0,1)\rangle_{\fqn})=s$, so that $|L_U|\leq q^{k-1}+\ldots+q^{t}-q^{s-1}-\ldots-q+1$, by \eqref{eq:card}, \eqref{eq:pesicard} and \eqref{eq:pesivett}. The equivalence between i) and ii) follows from the fact that $|L_U|= q^{k-1}+\ldots+q^{t}-q^{s-1}-\ldots-q+1$ if and only if $w_{L_U}(P)=1$ for every $P \in L_U \setminus\{\langle(1,0)\rangle_{\fqn},\langle(0,1)\rangle_{\fqn}\}$.
Consider a nonzero $\lambda \in S\cdot S^{-1} \cap T\cdot T^{-1}$, then $\lambda=x'/x=y'/y$ for some nonzero $x,x' \in T$ and $y,y' \in S$, that is $(x',y')=\lambda(x,y)$. Hence ii) and iii) are equivalent.
Finally, suppose that $S\cdot S^{-1} \cap T \cdot T^{-1}=\F_q$.
By contradiction, assume the existence of $\alpha \in \F_{q^n}^*$ such that $\dim_{\fq}(S \cap \alpha T) \geq 2$.
Let $a,c \in S \cap \alpha T$ such that $a$ and $c$ are $\fq$-linearly independent.
Moreover, since $a,c \in \alpha T$, there exist $b,d \in T\setminus\{0\}$ such that $a=\alpha b$ and $c=\alpha d$.
Therefore, $ca^{-1}=db^{-1}$ is in $S\cdot S^{-1}\cap T\cdot T^{-1}$ and hence $ca^{-1} \in \fq$, by assumption.
This yields to a contradiction.
Now, assume that $\dim_{\fq}(S \cap \alpha T) \leq 1$ for each $\alpha \in \F_{q^n}^*$.
Let $\eta \in S\cdot S^{-1} \cap T \cdot T^{-1} \setminus\{0\}$. Then there exist $a,c \in S\setminus\{0\}$ and $b,d \in T\setminus\{0\}$ such that $\eta=ac^{-1}=db^{-1}$.
Then $a,c \in S \cap \beta T$ where $\beta=a/d=c/b$. Since $\dim_{\fq}(S \cap \beta T)\leq 1$, it follows that $\eta=ac^{-1} \in \F_q$.
\end{proof}

\begin{remark}
The condition on the subspaces $S$ and $T$ in the above result arises also in the context of cyclic subspace codes, see \cite{Roth}, which are possible candidates for large codes with efficient encoding and decoding algorithms.
\end{remark}

\begin{corollary}\label{cor:UasUST}
Let $L$ be an $\fq$-linear set of rank $k$ in $\PG(1,q^n)$ for which there exist two distinct points $P,Q \in L$ such that
$w_{L}(P)=s$, $w_{L}(Q)=t$, $s+t=k\leq n$ and $s,t>1$. 
If $P$ and $Q$ are the only points of $L$ of weight greater than one, then $s\leq \frac{n}2$ and $t\leq \frac{n}2$.
In particular, if $s+t=n$ $n$ is even and then $s=t=\frac{n}2$.
\end{corollary}
\begin{proof}
By Proposition \ref{th:UasUST}, up to equivalence, we can suppose $L=L_U$ with $U=T\times S$, where $S$ and $T$ are $\fq$-subspaces of $\fqn$ with $\dim_{\fq}(S)=s>1$, $\dim_{\fq}(T)=t>1$ and $s+t=k$.
Also, $S$ and $T$ satisfy iii) of Theorem \ref{th:SS-1}.

Now, suppose that $s> n/2$ and consider $W=S \times S$. Then the $\fq$-linear set $L_W$ has rank $2s>n$ and hence $|L_W|=q^n+1$. 
On the other side, $|L_W|=|S \cdot S^{-1}|+1$. So that $S \cdot S^{-1}=\fqn$ and by Theorem \ref{th:SS-1} it follows $T\cdot T^{-1}=\fq$, which is a contradiction since $t>1$.
Similar arguments can be performed on $T$.
The last part is then a trivial consequence.
\end{proof}

In \cite{DeBoeckVdV20} the authors study the weight distribution of linear sets in $\PG(1,q^5)$. In particular, they prove that linear sets of rank $5$ in $\PG(1,q^5)$ containing one point of weight $3$, one point of weight $2$ and all the others of weight one do not exist.
As a consequence of Corollary \ref{cor:UasUST} we prove the following generalization of this result when $n$ is odd.

\begin{theorem}
Let $L$ be an $\fq$-linear set of $\PG(1,q^n)$ with exactly two points $P$ and $Q$ of weight greater than one, then 
\[ w_L(P)\leq \frac{n}2 \,\,\text{and}\,\,w_L(Q)\leq \frac{n}2. \]
In particular, if $n$ is odd then $w_L(P)+w_L(Q)<n$.
Moreover, if $w_L(P)+w_L(Q)=n$ then $n$ is even and $w_L(P)=w_L(Q)=\frac{n}2$.
\end{theorem}
\begin{proof}
Let $L=L_U$ and let $P=\langle w\rangle_{\F_{q^n}}$ and $Q=\langle v \rangle_{\F_{q^n}}$ such that $w_L(P)=t$ and $w_L(Q)=s$.
Let $W=\langle w\rangle_{\F_{q^n}} \cap U$ and $V=\langle v\rangle_{\F_{q^n}} \cap U$.
Then $L_{\overline{U}}$, where $\overline{U}=W+V$ is and $\fq$-linear set of $\PG(1,q^n)$ of rank $h=\dim_{\fq}(W)+\dim_{\fq}(V)=t+s$ and by construction $w_{L_{\overline{U}}}(P)=t$ and $w_{L_{\overline{U}}}(Q)=s$.
Since $\overline{U}$ is contained in $U$, for any $T \in L_{\overline{U}}$ we have $w_{L_{\overline{U}}}(T)\leq w_{L_U}(T)$.
Hence, $P$ and $Q$ are the only points of $L_{\overline{U}}$ of weight greater than one, so by Corollary \ref{cor:UasUST}, we obtain $t\leq \frac{n}2$ and $s\leq \frac{n}2$.
\end{proof}

Now, our aim is to construct examples of $\fq$-linear sets of $\PG(1,q^n)$ of rank $k$ attaining the upper bound in \eqref{eq:bound}. 

\begin{theorem}
Let $\F_{q^t}$ be a subfield of $\fqn$ with $1<t<n$ and let $n=t\ell$.
There exist $\fq$-linear sets of rank $k$ in $\PG(1,q^n)$ with one point of weight $t$, one point of weight $s$ and all others of weight one for the following values of $n$, $k$ and $s$:
\begin{itemize}
    \item $n$ even, $k=t+s$ and any $s \in \{1,\ldots,\frac{n}2\}$;
    \item $n$ odd, $k=t+s$ and any $s \in \{1,\ldots,\frac{n-t}2\}$.
\end{itemize}
In particular, if $n$ is even and $t=\frac{n}2$, there exist $\fq$-linear sets of rank $n$ in $\mathrm{PG}(1,q^n)$ with exactly two points of weight greater than one and both have weight $\frac{n}2$.
\end{theorem}
\begin{proof}
Let $T=\F_{q^t}$ and let $S$ be an $\fq$-subspace of $\fqn$ such that $S$ is scattered with respect to $\F_{q^t}$, i.e. $S$ defines in $\PG(\fqn,\F_{q^t})=\PG(\ell-1,q^t)$ a scattered $\fq$-linear set.
This means that $\dim_{\fq}(S\cap \alpha \F_{q^t})\leq 1$ for every $\alpha \in \fqn^*$.
So, by Theorem \ref{th:SS-1}, $U=\F_{q^t}\times S$ defines in $\PG(1,q^n)$ an $\fq$-linear set of rank $k=t+s$, where $s=\dim_{\fq}(S)$, having one point of weight $t$, one point of weight $s$ and all the others of weight one.
If $n$ is even, by Theorem \ref{th:existscattered} there exist scattered $\fq$-linear sets of rank $s$ in $\PG (\ell-1,q^t)$ for any $s\in \{1,\ldots,\frac{n}2\}$.
If $n$ is odd, by Example \ref{ex:scattered} there exist scattered $\fq$-linear sets of rank $s$ in $\PG (\ell-1,q^t)$ for any $s\in \{1,\ldots,\frac{n-t}2\}$.
\end{proof}

Now we are able to determine a canonical form for $\fq$-subspaces representing linear sets of rank $n=2t$ having two points of weight $t$ and all the other points of weight one.

\begin{theorem} \label{th:relationfandg}
Let $L$ be an $\fq$-linear set of $\PG(1,q^{2t})$ of rank $n=2t$ admitting two points $P$ and $Q$ of weight $t$. Then $L$ is equivalent to an $\fq$-linear set $L_U$ where
\[ U=T_{g,\eta} \times S_{f,\xi}=\{ (v+\eta g(v),u+\xi f(u)) \colon u,v \in \F_{q^t} \}, \]
with $\eta,\xi \in \F_{q^{2t}}\setminus\F_{q^t}$, $\xi^2=a\xi+b$, $\eta=A\xi+B$ with $a,b,A,B \in \F_{q^t}$ and $f(x),g(x) \in \mathcal{L}_{t,q}$.
Also, $P$ and $Q$ are the only points of $L$ with weight greater than one if and only if 
\begin{equation} \label{eq:relationfandg}
f(\alpha_0 v)+f(\alpha_1 A b g(v))+f(\alpha_0 B g(v))=\alpha_1 v+\alpha_0 A g(v)+\alpha_1 A a g(v)+\alpha_1 B g(v),
\end{equation}
has at most $q$ solutions in $v$ for every $\alpha_0,\alpha_1 \in \F_{q^t}$ with $(\alpha_0,\alpha_1)\ne(0,0)$.
\end{theorem}
\begin{proof}
By Proposition \ref{th:UasUST}, $L$ is $\mathrm{PGL}(2,q^n)$-equivalent to $L_U$ where $U=S\times T$, $S$ and $T$ are $\fq$-subspaces of $\fqn$ and $\dim_{\fq}(S)=\dim_{\fq}(T)=t$.
Let $\eta,\xi \in \F_{q^{2t}}\setminus \F_{q^t}$. Arguing as in the last part of Proposition \ref{th:UasUST} there exist $\lambda_1$ and $\lambda_2$ in $\fqn^*$ such that
\[ \lambda_1S \cap \xi \F_{q^t}=\lambda_2 T \cap \eta \F_{q^t} =\{0\}. \]
This means that 
\[ \lambda_1 S=S_{f,\xi}=\{ u+\xi f(u) \colon u \in \F_{q^t} \}, \]
\[ \lambda_2 T=T_{g,\eta}=\{ u+\eta g(u) \colon u \in \F_{q^t} \}, \]
for some $q$-polynomials $f(x)$ and $g(x)$ in $\mathcal{L}_{t,q}$.
Consider $w \in S_{f,\xi} \cap \alpha T_{g,\eta}$, for some $\alpha \in \F_{q^{2t}}^*$.
Then there exist $u$ and $v \in \F_{q^t}$ such that
\begin{equation}\label{eq:uvfg} w=u+\xi f(u)=\alpha(v+\eta g(v)). \end{equation}
Since $\{1,\xi\}$ is an $\F_{q^t}$-basis of $\F_{q^{2t}}$, $\xi^2=a\xi+b$, $\eta=A\xi+B$, we obtain $\alpha=\alpha_0+\alpha_1\xi$ for some $\alpha_0,\alpha_1 \in \F_{q^t}$, and \eqref{eq:uvfg} reads as follows
\[ 
\left\{
\begin{array}{ll}
u=\alpha_0 v +\alpha_1 A b g(v)+\alpha_0 B g(v),\\
f(u)=\alpha_1 v +\alpha_0 A g(v)+\alpha_1 A a g(v)+\alpha_1 B g(v).
\end{array}
\right.
\]
By replacing the first equation of the above system into the second one, we obtain
\[ f(\alpha_0 v) + f(\alpha_1 A b g(v))+f(\alpha_0 B g(v))=\alpha_1 v +\alpha_0 A g(v)+\alpha_1 A a g(v)+\alpha_1 B g(v). \]
The assertion then follows from Theorem \ref{th:SS-1}.
\end{proof}

As a corollary we obtain several examples.

\begin{corollary}\label{cor:ex1}
Let $g(x)=\mu x$ with $\mu \in \F_{q^t}$ and let $f(x) \in \mathcal{L}_{t,q}$ be a $q$-polynomial.
Let
\[ U=T_{g,\eta} \times S_{f,\xi}=\{ (v+\eta \mu v,u+\xi f(u)) \colon u,v \in \F_{q^t} \}\subseteq \F_{q^{2t}}\times \F_{q^{2t}}, \]
with $\xi,\eta \in \F_{q^{2t}}\setminus\F_{q^t}$. 
Then $L_U$ has exactly two points of weight greater than one if and only if $f(x)$ is a scattered $q$-polynomial in $\mathcal{L}_{t,q}$.
\end{corollary}
\begin{proof}
First we observe that we may assume $\mu=0$, i.e. $T=\F_{q^t}$. Indeed, let $\delta=1+\mu \eta$, then 
$U$ is $\mathrm{GL}(2,q^n)$-equivalent to
\[ U'=\delta^{-1}T_{g,\eta}\times S_{f,\xi}=\F_{q^t}\times S_{f,\xi}=T_{0,\eta}\times S_{f,\xi}, \]
and $\dim_{\fq}(S_{f,\xi} \cap \alpha T_{g,\eta})\leq 1$ for every $\alpha \in \F_{q^{2t}}$ if and only if $\dim_{\fq}(S_{f,\xi} \cap \alpha T_{0,\eta})\leq 1$ for every $\alpha \in \F_{q^{2t}}$.
Hence, let $g(x)=0$, applying Theorem \ref{th:relationfandg} Equation \eqref{eq:relationfandg} reads
\[ f(\alpha_0 v)=\alpha_1 v. \]
It has at most $q$ solutions in $\F_{q^{t}}$ for every $\alpha_0,\alpha_1 \in \F_{q^t}$ if and only if 
\[\dim_{\fq}(\ker(f(x)-\beta x))\leq 1\] for every $\beta \in \F_{q^t}$, that is $f(x)$ is a scattered $q$-polynomial in $\mathcal{L}_{t,q}$.
\end{proof}


Applying Theorem \ref{th:relationfandg} to the case $\eta=\xi$ and, $f(x)$ and $g(x)$ are monomials we obtain the following result.

\begin{corollary}\label{cor:ex2}
Let $s$ be a positive integer coprime with $t$, $f(x)=x^{q^s}$ and $g(x)=\mu x^{q^s}$ with $\mu \in \F_{q^t}^*$ such that $\N_{q^t/q}(\mu) \neq 1$ and $\N_{q^{t}/q}(-\xi^{q^t+1}\mu)\neq (-1)^t$, 
with $\xi \in \F_{q^{2t}}\setminus\F_{q^t}$.
Let
\[ U=T_{g,\xi} \times S_{f,\xi}=\{ (v+\xi \mu v^{q^s},u+\xi u^{q^s}) \colon u,v \in \F_{q^t} \}. \]
Then $L_U$ has exactly two points of weight greater than one.
\end{corollary}
\begin{proof}
Let $\xi^2=a\xi+b$, for some $a,b \in \F_{q^t}$, and note that $b=-\xi^{q^t+1}$.
Equation \eqref{eq:relationfandg} reads
\begin{equation}\label{eq:pseudoradici} \alpha_0^{q^s}v^{q^s}+b^{q^s}\alpha_1^{q^s}\mu^{q^s}v^{q^{2s}}-\alpha_1 v-\alpha_0 \mu v^{q^s}-\alpha_1 a \mu v^{q^s}=0, \end{equation}
and the left hand side of \eqref{eq:pseudoradici} is either a $q^s$-polynomial of $q^s$-degree at most $2$ or the zero polynomial.
If $\alpha_1=0$, then \eqref{eq:pseudoradici} becomes
\[ \alpha_0^{q^s}v^{q^s}-\alpha_0 \mu v^{q^s}=0, \]
and this implies that $\alpha_0^{q^s-1}=\mu$, which is not possible because of $\N_{q^t/q}(\mu)\ne 1$.
If $\alpha_1\ne 0$, the left hand side of \eqref{eq:pseudoradici} is a nonzero $q^s$-polynomial of $q^s$-degree two and if it admits $q^2$ roots in $v$ then Theorem \ref{Gow} implies $\N_{q^{t}/q}(\alpha_1^{1-q^s}b^{-q^s}\mu^{-q^s})=(-1)^t$, that is $\N_{q^t/q}(b\mu)=(-1)^t$ which is a contradiction.
\end{proof}

\subsection{Equivalence issue}

In this section we study the equivalence issue for the $\fq$-subspaces introduced in Corollaries \ref{cor:ex1} and \ref{cor:ex2}.

We will need the following result.

\begin{lemma}\label{lem:decomposition}
Let $\left(\begin{array}{cc} a_{11} & a_{12}\\ a_{21} & a_{22} \end{array}\right)\in \mathrm{GL}(2,q^t)$ and let $a,b \in \F_{q^t}$ such that $x^2-ax-b\in \F_{q^t}[x]$ is irreducible. 
Then there exist $\mu_0,\mu_1,\alpha,\beta \in \F_{q^t}$ such that $(\mu_0,\mu_1)\ne (0,0)$, $\alpha \ne 0$ and
\[ \left(\begin{array}{cc} a_{11} & a_{12}\\ a_{21} & a_{22} \end{array}\right) =\left(\begin{array}{cc} \mu_0 & \mu_1 b\\ \mu_1 & \mu_0+a \mu_1 \end{array}\right)\left(\begin{array}{cc} 1 & \beta\\ 0 & \alpha \end{array}\right), \]
i.e.\ $\left(\begin{array}{cc} a_{11} & a_{12}\\ a_{21} & a_{22} \end{array}\right)$ is the product of an element of the Singer subgroup of $\mathrm{GL}(2,q^t)$ determined by $x^2+ax-b$ and of an upper triangular matrix.
\end{lemma}
\begin{proof}
Recall that $\mathrm{GL}(2,q^t)$ is isomorphic to the following subgroup of $\mathcal{L}_{2t,q}$
\[ \mathcal{G}=\{ f(x)=Ax+Bx^{q^{t}} \colon A,B \in \F_{q^{2t}}\,\,\text{and}\,\,f(x)\,\,\text{is invertible} \}/(x^{q^{2t}}-x). \]
Let $f(x)=Ax+Bx^{q^{t}}\in \mathcal{G}$ and note that, since $f(x)$ is invertible, $\N_{q^{2t}/q^t}(A)\ne \N_{q^{2t}/q^t}(B)$ which implies that $A+B \ne 0$, so we can write
\[f(x)=(A+B)\left(\frac{A}{A+B}x+\frac{B}{A+B}x^{q^{t}}\right)=\tau_{A+B}\circ \overline{f}(x),\] 
where $\tau_{A+B}(x)=(A+B)x$ and $\overline{f}(x)=\frac{A}{A+B}x+\frac{B}{A+B}x^{q^{t}}$.
Let $\mu=A+B$, $\eta \in \F_{q^{2t}}\setminus \F_{q^t}$ be a root of $x^2-ax-b$ and $\mu=\mu_0+\mu_1\eta$ with $\mu_0,\mu_1 \in \F_{q^t}$.
Then the matrix of $\mathrm{GL}(2,q^t)$ associated with $\overline{f}$ with respect to the $\F_{q^t}$-basis $(1,\eta)$ is
\[ \left(\begin{array}{cc} 1 & \beta\\ 0 & \alpha \end{array}\right) \]
where $\alpha,\beta \in \F_{q^t}$ such that
\[ \alpha \eta +\beta=\frac{A}{A+B}\eta + \frac{B}{A+B}\eta^{q^t}, \]
and the matrix associated with $\tau_{A+B}(x)$ is
\[ \left(\begin{array}{cc} \mu_0 & \mu_1 b\\ \mu_1 & \mu_0+a \mu_1 \end{array}\right). \]
Since $f(x)=\tau_{A+B} \circ \overline{f}(x)$, the assertion follows.
\end{proof}

\begin{theorem}
Let $n=2t$ be an even positive integer and let $\xi,\eta \in \fqn\setminus \F_{q^t}$. 
Let $S_{f,\xi}=\{z+\xi f(z) \colon z \in \F_{q^t}\}$ and $S_{g,\eta}=\{z+\eta g(z) \colon z \in \F_{q^t}\}$ where $f(z)$ and $g(z)$ are scattered $q$-polynomials in $\mathcal{L}_{t,q}$.
If $U=\F_{q^t}\times S_{f,\xi}$ and $U'=\F_{q^t}\times S_{g,\eta}$
are $\mathrm{\Gamma L}(2,q^{2t})$-equivalent then the $\fq$-subspaces $U_f=\{(x,f(x))\colon x \in \F_{q^t}\}$ and $U_g=\{(x,g(x))\colon x \in \F_{q^t}\}$ are $\mathrm{\Gamma L}(2,q^{t})$-equivalent in $\F_{q^t}\times \F_{q^t}$.
Conversely, if $U_f=\{(x,f(x))\colon x \in \F_{q^t}\}$ and $U_g=\{(x,g(x))\colon x \in \F_{q^t}\}$ are $\mathrm{\Gamma L}(2,q^{t})$-equivalent then for every $\eta \in \fqn\setminus \F_{q^t}$ there exists $\xi \in \fqn\setminus \F_{q^t}$ such that $U=\F_{q^t}\times S_{f,\xi}$ and $U'=\F_{q^t}\times S_{g,\eta}$ are $\mathrm{\Gamma L}(2,q^{2t})$-equivalent.
\end{theorem}
\begin{proof}
Since $f(x)$ is not a constant  polynomial, then it is easy to see that $\F_{q^t}\ne \mu (S_{f,\xi})^\sigma$ for any $\mu \in \fqn^*$ and $\sigma \in \mathrm{Aut}(\fqn)$. So, Lemma \ref{lem:criteria} implies that $U$ is $\mathrm{\Gamma L}(2,q^{2t})$-equivalent to $U'$ if and only if there exist $\mu \in \fqn^*$ and $\sigma \in \mathrm{Aut}(\fqn)$ such that
\[ S_{g,\eta}=\mu (S_{f,\xi})^\sigma. \]
Assume that such $\mu$ and $\sigma$ exist, so for every $z \in \F_{q^t}$ there exists $y \in \F_{q^t}$ such that
\[ z+\eta g(z)=\mu (y^\sigma+\xi^\sigma f(y)^\sigma). \]
Let $\eta^2=a\eta +b$ with $a,b \in \F_{q^t}$, $\xi^\sigma=\alpha \eta+\beta$ with $\alpha,\beta \in \F_{q^t}$ and let $\mu=\mu_0+\mu_1\eta$ where $\mu_0,\mu_1 \in \F_{q^t}$.
Then the above equality yields
\[
\left\{
\begin{array}{ll}
z=\mu_0 y^\sigma +(\mu_0\beta +\mu_1\alpha b)f(y)^\sigma,\\
g(z)=\mu_1 y^\sigma+(\mu_0\alpha+\mu_1 \beta+a\mu_1\alpha)f(y)^\sigma,
\end{array}
\right.
\]
i.e.
\begin{equation}\label{eq:equivscatteredmatrix}
\left( 
\begin{array}{cc}
    \mu_0 & \mu_0\beta +\mu_1\alpha b \\
    \mu_1 & \mu_0\alpha + (\beta +a\alpha)\mu_1
\end{array}
\right)
\left( 
\begin{array}{cc}
    y^\sigma \\
    f(y)^\sigma
\end{array}
\right)
= 
\left( 
\begin{array}{cc}
    z \\
    g(z)
\end{array}
\right).
\end{equation}
Since the polynomial $x^2-ax-b \in \F_{q^t}[x]$ is irreducible and $\alpha\ne 0$, the matrix of the previous equality is invertible and hence $U_f$ and $U_g$ are $\mathrm{\Gamma L}(2,q^t)$-equivalent.
Note that 
\[  
\left( 
\begin{array}{cc}
    \mu_0 & \mu_0\beta +\mu_1\alpha b \\
    \mu_1 & \mu_0\alpha + (\beta +a\alpha)\mu_1
\end{array}
\right)=
\left( 
\begin{array}{cc}
    \mu_0 &  \mu_1 b \\
    \mu_1 & \mu_0+a\mu_1
\end{array}
\right)
\left( 
\begin{array}{cc}
    1 & \beta \\
    0 & \alpha
\end{array}
\right).
\]
Conversely, suppose that $U_f$ and $U_g$ are $\mathrm{\Gamma L}(2,q^{t})$-equivalent, then for every $x \in \F_{q^t}$ there exists $y \in \F_{q^t}$ such that
\[  
\left( 
\begin{array}{cc}
    a_{11} & a_{12} \\
    a_{21} & a_{22}
\end{array}
\right)
\left( 
\begin{array}{cc}
    x^\sigma \\
    f(x)^\sigma
\end{array}
\right)
=
\left( 
\begin{array}{cc}
    y \\
    g(y)
\end{array}
\right),
\]
where $\left( 
\begin{array}{cc}
    a_{11} & a_{12} \\
    a_{21} & a_{22}
\end{array}
\right) \in \mathrm{GL}(2,q^{t})$ and $\sigma \in \mathrm{Aut}(\F_{q^t})$.
By Lemma \ref{lem:decomposition}, for a fixed element $\eta \in \fqn\setminus \F_{q^t}$ such that $\eta^2=a\eta+b$, there exist $\mu_0,\mu_1,\alpha,\beta \in \F_{q^t}$ such that
\[ \left( 
\begin{array}{cc}
    a_{11} & a_{12} \\
    a_{21} & a_{22}
\end{array}
\right) = 
\left( 
\begin{array}{cc}
    \mu_0 &  \mu_1 b \\
    \mu_1 & \mu_0+a\mu_1
\end{array}
\right)
\left( 
\begin{array}{cc}
    1 & \beta \\
    0 & \alpha
\end{array}
\right).
\]
Let $\xi=(\alpha \eta+\beta)^{\sigma^{-1}}$ and $\mu=\mu_0+\mu_1\eta$, then by \eqref{eq:equivscatteredmatrix} we have
\[ S_{g,\eta}=\mu (S_{f,\xi})^\sigma, \]
hence $U$ and $U'$ are $\mathrm{\Gamma L}(2,q^n)$-equivalent and the assertion then follows.
\end{proof}

\begin{remark}
By the previous theorem, we can consider scattered $q$-polynomials in $\mathcal{L}_{t,q}$ defining $\Gamma\mathrm{L}(2,q^t)$-inequivalent subspaces to obtain $\mathrm{\Gamma L}(2,q^{2t})$-inequivalent subspaces as in the above result.
In Table \ref{scattpoly} we resume all the $q$-polynomials defining $\Gamma\mathrm{L}(2,q^t)$-inequivalent subspaces of $\F_{q^t}\times \F_{q^t}$.
\end{remark}

\begin{table}[htp]
\tabcolsep=0.2 mm
\begin{tabular}{|c|c|c|c|c|}
\hline
\hspace{0.2cm}$t$\hspace{0.2cm} & $f(x)$ & \mbox{conditions} & \mbox{references} \\ \hline
& $x^{q^s}$ & $\gcd(s,t)=1$ & \cite{BL2000} \\ \hline
 & $x^{q^s}+\delta x^{q^{s(t-1)}}$ & $\begin{array}{cc} \gcd(s,t)=1,\\ \mathrm{N}_{q^t/q}(\delta)\neq 1 \end{array}$ & \cite{LunPol2000,LMPT2015}\\ \hline
$2\ell$ & $x^{q^s}+x^{q^{s(\ell-1)}}+\delta^{q^\ell+1}x^{q^{s(\ell+1)}}+\delta^{1-q^{2\ell-1}}x^{q^{s(2\ell-1)}}$ & $\begin{array}{cc} q \hspace{0.1cm} \text{odd}, \\ \mathrm{N}_{q^{2\ell}/q^\ell}(\delta)=-1,\\ \gcd(s,t)=1 \end{array}$ & \cite{BZZ,LMTZ,LZ2,NSZ,ZZ}\\ \hline
$6$ & $x^q+\delta x^{q^{4}}$  &  $\begin{array}{cc} q>4, \\ \text{certain choices of} \, \delta \end{array}$ & \cite{CMPZ,BCsM,PZ2019} \\ \hline
$6$ & $x^{q}+x^{q^3}+\delta x^{q^5}$ & $\begin{array}{cccc}q \hspace{0.1cm} \text{odd}, \\ \delta^2+\delta =1 \end{array}$
 & \cite{CsMZ2018,MMZ} \\ \hline
$8$ & $x^{q}+\delta x^{q^5}$ & $\begin{array}{cc} q\,\text{odd},\\ \delta^2=-1\end{array}$ & \cite{CMPZ} \\ \hline
\end{tabular}
\caption{Known examples of scattered polynomials in $\mathcal{L}_{t,q}$}
\label{scattpoly}
\end{table}

We now show that the subspaces of Corollary \ref{cor:ex1} cannot be equivalent to the subspaces of Corollary \ref{cor:ex2}.

\begin{proposition}
Let $n=2t$ be an even positive integer, $\xi,\eta \in \fqn\setminus \F_{q^t}$.
Let $f(x)$ be a scattered $q$-polynomial in $\mathcal{L}_{t,q}$, $s$ be a positive integer coprime with $t$, $\mu \in \F_{q^t}^*$ such that $\N_{q^t/q}(\mu)\neq 1$ and $\N_{q^t/q}(-\xi^{q^t+1}\mu)\ne (-1)^t$.
Then the $\fq$-subspaces \[U=\F_{q^t}\times S_{f,\eta}\,\,\,\,\,\, \text{and}\,\,\,\,\,\, U'=T_{x^{q^s},\xi}\times S_{\mu x^{q^s},\xi}\] 
are $\Gamma\mathrm{L}(2,q^n)$-inequivalent.
\end{proposition}
\begin{proof}
By contradiction, assume that $U$ and $U'$ are $\Gamma\mathrm{L}(2,q^n)$-equivalent.
From Lemma \ref{lem:criteria}, we have that either $T_{x^{q^s},\xi}=\lambda \F_{q^t}^\rho$, or $S_{\mu x^{q^s},\xi}=\lambda \F_{q^t}^\rho$, for some $\lambda \in \F_{q^n}^*$ and $\rho \in \mathrm{Aut}(\fqn)$, which is a contradiction since $T_{x^{q^s},\xi}$ and $S_{\mu x^{q^s},\xi}$ are scattered $\fq$-subspaces in $\fqn=V(2,q^t)$.
\end{proof}

Now, we investigate the $\mathrm{\Gamma L}(2,q^n)$-equivalence of the examples in Corollary \ref{cor:ex2}.

\begin{lemma} \label{lemma:singolosubspace}
Let $n=2t$ be an even positive integer and $t\geq 3$.
Let $\xi,\eta \in \F_{q^{2t}}\setminus \F_{q^t}$.
Let $s,s'$ be positive integers such that $\gcd(s,t)=\gcd(s',t)=1$.
Let $\mu,\overline{\mu}\in \F_{q^t}^*$ and consider the $\fq$-subspaces of $\F_{q^{2t}}$
\[ S_{\mu x^{q^s},\xi}=\{ u+\xi \mu u^{q^s} \colon u \in \F_{q^t} \} \]
and
\[ S_{\overline{\mu} x^{q^{s'}},\eta}=\{ v+\eta \overline{\mu} v^{q^{s'}} \colon v \in \F_{q^t} \}. \]
There exist $\lambda \in \fqn^*$ and $\sigma \in \mathrm{Aut}(\fqn)$ such that
\begin{equation}\label{eq:equivalencemonomials}
    S_{\mu x^{q^s},\xi}=\lambda (S_{\overline{\mu} x^{q^{s'}},\eta})^\sigma
\end{equation}
if and only if one of the following condition holds
\begin{itemize}
    \item $s\equiv -s'\pmod{t}$, $\xi=\frac{\eta^\sigma+Aa}{A}$, $\overline{\mu}=\frac{1}{c \mu^{q^{-s}\sigma^{-1}} A^{\sigma^{-1}}b^{\sigma^{-1}}}$ and $B=-Aa$, where $c\in \F_{q^t}$ such that $\N_{q^t/q}(c)=1$;
    \item $s\equiv s'\pmod{t}$, $B=0$, $\xi =\frac{\eta^{\sigma}}{A}$ and $\overline{\mu}=\frac{\mu^{\sigma^{-1}}c}{A^{\sigma^{-1}}}$, where $c\in \F_{q^t}$ such that $\N_{q^t/q}(c)=1$,
\end{itemize}
where $\xi^2=a\xi+b$ and $\eta^\sigma=A\xi+B$ with $a,b,A,B \in \F_{q^t}$.
\end{lemma}
\begin{proof}
Let $\lambda=\lambda_0+\lambda_1 \xi$, with $\lambda_0,\lambda_1 \in \F_{q^t}$, then \eqref{eq:equivalencemonomials} is satisfied if and only if for every $v \in \F_{q^t}$ there exists $u \in \F_{q^t}$ such that 
\[ v+\xi \mu v^{q^s}=(\lambda_0+\lambda_1\xi)(u^\sigma +\overline{\mu}^\sigma u^{\sigma q^{s'}} \eta^{\sigma}), \]
that is 
\[
\left\{
\begin{array}{ll}
v= \lambda_0 u^{\sigma}+(\lambda_0 \overline{\mu}^\sigma B + \lambda_1 \overline{\mu}^\sigma A b)u^{\sigma q^{s'}},\\
\mu v^{q^s}=\lambda_1 u^{\sigma}+(\lambda_0\overline{\mu}^\sigma A + \lambda_1 \overline{\mu}^\sigma A a+\lambda_1\overline{\mu}^\sigma B)u^{\sigma q^{s'}}.
\end{array}
\right.
\]
This means that
\begin{small}
\begin{equation}\label{eq:condpolmon} \mu\lambda_0^{q^s}u^{\sigma q^s}+\mu (\lambda_0 \overline{\mu}^\sigma B + \lambda_1 \overline{\mu}^\sigma A b)^{q^s} u^{\sigma q^{s'+s}}=\lambda_1 u^{\sigma}+(\lambda_0\overline{\mu}^\sigma A + \lambda_1 \overline{\mu}^\sigma A a+\lambda_1\overline{\mu}^\sigma B)u^{\sigma q^{s'}}, \end{equation}
\end{small}
for every $u \in \F_{q^t}$ and hence it can be seen as a polynomial identity in $u^\sigma$.

First we show that $\{0 \pmod{t},s \pmod{t},s' \pmod{t},s+s'\pmod{t}\}$ cannot have size $4$.
Indeed, in such a case by \eqref{eq:condpolmon} we easily get $\lambda_0=\lambda_1=0$, a contradiction.
Now, since $\gcd(s,t)=\gcd(s',t)=1$ and $t\geq 3$, only one of the following cases occurs:
\begin{itemize}
    \item $s \equiv -s' \pmod{t}$;
    \item $s \equiv s' \pmod{t}$.
\end{itemize}
Suppose that $s \equiv -s' \pmod{t}$. From the coefficient
in \eqref{eq:condpolmon} of $q$-degree $s$ we have $\lambda_0=0$, and from the remaining coefficients we obtain $\lambda_1 \overline{\mu}^\sigma (Aa+B)=0$ and $\lambda_1=\mu\lambda_1 \overline{\mu}^\sigma A b$.
So, $B=-Aa$ and $\overline{\mu}^\sigma=\frac{1}{c \mu^{q^{-s}} A b}$, where $c\in \F_{q^t}$ such that $\N_{q^t/q}(c)=1$, which imply $\eta^\sigma=A \xi -Aa$.
The assertion then follows.
Suppose now that $s \equiv s' \pmod{t}$.
From \eqref{eq:condpolmon} we get $\lambda_1=B=0$ and $\mu \lambda_0^{q^s}=\lambda_0\overline{\mu}^\sigma A$, which imply $\xi=\frac{\eta^{\sigma}}A$ and $\overline{\mu}^\sigma=\frac{\mu \lambda_0^{q^s-1}}{A}$ and the assertion follows.
\end{proof}

The above lemma allows us to prove the next classification result.

\begin{theorem}
Let $\xi,\eta \in \fqn \setminus \F_{q^t}$, with $\xi^2=a\xi+b$ and $t\geq 3$.
Let $\mu_1,\mu_2 \in \F_{q^n}^*$ such that $\N_{q^t/q}(\mu_i)\neq 1$, $\N_{q^t/q}(-\xi^{q^t+1}\mu_1 )\ne (-1)^t$, $\N_{q^t/q}(-\eta^{q^t+1}\mu_2)\ne (-1)^t$, $s,s'$ be positive integers coprime with $t$.
The $\fq$-subspaces $U=S_{x^{q^s},\xi} \times T_{\mu_1 x^{q^{s}},\xi}$ and $U'= S_{x^{q^{s'}},\eta} \times T_{\mu_2 x^{q^{s'}},\eta}$ are $\mathrm{\Gamma L}(2,q^n)$-equivalent through the map $\varphi \in \mathrm{\Gamma L}(2,q^n)$ with companion automorphism $\sigma \in \mathrm{Aut}(\fqn)$ if and only if one of the following conditions holds:

\begin{enumerate}
    \item[\bf{I)}] $s\equiv -s'\pmod{t}$, $\xi=\frac{\eta^\sigma+Aa}{A}$, $B=-Aa$ and either
    \begin{itemize}
      \item[\bf{I.1)}]  $\mu_2=\frac{1}{c_1 A^{\sigma^{-1}}b^{\sigma^{-1}}}$ and $\mu_1=\frac{1}{c_2A^{q^s}b^{q^s}}$, where $c_1,c_2\in \F_{q^t}$ such that $\N_{q^t/q}(c_i)=1$, or
      \item[\bf{I.2)}] $\N_{q^t/q}(Ab)=1$ and $\mu_2=\frac{1}{c\mu_1^{q^{-s}\sigma^{-1}}}$, where $c\in \F_{q^t}$ such that $\N_{q^t/q}(c)=1$.
    \end{itemize}
\item[\bf{II)}] $s\equiv s'\pmod{t}$,  $\xi =\frac{\eta^{\sigma}}{A}$, $B=0$ and either
    \begin{itemize}
     \item[\bf{II.1)}] $\mu_2=\frac{c_1}{A^{\sigma^{-1}}}$and $\mu_1=\frac{A}{c_2}$, where $c_1,c_2\in \F_{q^t}$ such that $\N_{q^t/q}(c_i)=1$, or
     \item[\bf{II.2)}] $\N_{q^t/q}(A)=1$ and $\mu_2=\frac{\mu_1^{\sigma^{-1}}c}{A^{\sigma^{-1}}}$, where $c\in \F_{q^t}$ such that $\N_{q^t/q}(c)=1$,
    \end{itemize}
\end{enumerate}
where $\eta^{\sigma}=A \xi +B$. 
\end{theorem}
\begin{proof}
Assume that $U$ and $U'$ are $\mathrm{\Gamma L}(2,q^n)$-equivalent.
By Lemma \ref{lem:criteria}, there exist $\lambda,\delta \in \fqn^*$ and $\sigma \in \mathrm{Aut}(\F_{q^n})$ such that either 
\[S_{x^{q^s},\xi}=\lambda (T_{\mu_2 x^{q^{s'}},\eta})^{\sigma}\,\, \text{and}\,\, T_{\mu_1 x^{q^{s}},\xi}=\delta (S_{x^{q^{s'}},\eta})^{\sigma}\] 
or 
\[S_{x^{q^s},\xi}=\delta (S_{x^{q^{s'}},\eta})^{\sigma}\,\, \text{and}\,\, T_{\mu_1 x^{q^{s}},\xi} =\lambda (T_{\mu_2 x^{q^{s'}},\eta})^{\sigma}.\] 
By Lemma \ref{lemma:singolosubspace}, the former case is verified if and only if one of the following conditions hold:

\begin{itemize}
    \item $s\equiv -s'\pmod{t}$, $B=-Aa$, $\xi=\frac{\eta^\sigma+Aa}{A}$,  $\mu_2=\frac{1}{c_1 A^{\sigma^{-1}}b^{\sigma^{-1}}}$ and $\mu_1=\frac{1}{c_2A^{q^s}b^{q^s}}$, where $c_1,c_2\in \F_{q^t}$ such that $\N_{q^t/q}(c_i)=1$;
    \item $s\equiv s'\pmod{t}$, $B=0$, $\xi =\frac{\eta^{\sigma}}{A}$, $\mu_2=\frac{c_1}{A^{\sigma^{-1}}}$, $\mu_1=\frac{A}{c_2}$, where $c_1,c_2\in \F_{q^t}$ such that $\N_{q^t/q}(c_i)=1$.
\end{itemize}
The latter case, by Lemma \ref{lemma:singolosubspace}, is verified if and only if one the following conditions hold:
\begin{itemize}
    \item $s\equiv -s'\pmod{t}$, $B=-Aa$, $\xi=\frac{\eta^\sigma+Aa}{A}$,  $\N_{q^t/q}(Ab)=1$ and $\mu_2=\frac{1}{c\mu_1^{q^{-s}\sigma^{-1}}}$, where $c\in \F_{q^t}$ such that $\N_{q^t/q}(c)=1$.;
    \item $s\equiv s'\pmod{t}$, $B=0$, $\xi =\frac{\eta^{\sigma}}{A}$, $\N_{q^t/q}(A)=1$, $\mu_2=\frac{\mu_1^{\sigma^{-1}}c}{A^{\sigma^{-1}}}$ where $c\in \F_{q^t}$ such that $\N_{q^t/q}(c)=1$.
\end{itemize}
This completes the proof.
\end{proof}

\begin{remark}
In particular by the previous theorem we obtain that the number of $\Gamma\mathrm{L}(2,q^{2t})$-inequivalent subspaces of form $U=T_{x^{q^s},\eta} \times S_{\mu x^{q^s},\xi}$ is at least $\varphi(t)/2$, where $\varphi$ is the Euler totient function.
\end{remark}

\begin{remark}
In the above result we avoid the case when $t=2$, since the $\fq$-linear sets of rank $4$ in $\PG(1,q^4)$ have been already classified, see \cite{BoPol,CsZ2018}.
In particular, Examples $B_{22}$ of \cite{BoPol} belong to the family of Corollary \ref{cor:ex1}, whereas Examples $C_{13}$ belong to the family of Corollary \ref{cor:ex2}.
\end{remark}

\subsection{Polynomial form}

In this subsection we determine a linearized polynomial $p(x)$ such that $L_U$ is $\mathrm{P\Gamma L}(2,q^n)$-equivalent to $L_p$, where $L_U$ is as in Corollaries \ref{cor:ex1} and \ref{cor:ex2}.

\begin{theorem}\label{th:pol2w}
Let $n=2t$ be an even positive integer. Consider $\F_{q^{t}}$ and let $S_{f,\xi}=\{z+\xi f(z) \colon z \in \F_{q^{n/2}}\}$, where $f(z)=\sum_{i=0}^{t-1} A_iz^{q^i} \in \mathcal{L}_{t,q}$ is a scattered polynomial.
Let \[ U=\F_{q^t} \times S_{f,\xi}. \] 
Then $L_U$ is $\mathrm{P}\Gamma\mathrm{L}(2,q^n)$-equivalent to $L_p$, with 
\[ p(x)=\mathrm{Tr}_{q^n/q^{t}}\left( \frac{A_0+\epsilon^{q^{t}}}{\epsilon^{q^{t}}-\epsilon}x+ \sum_{i=1}^{t-1} \frac{A_i}{\epsilon^{q^{i+t}}-\epsilon^{q^i}} x^{q^i} \right)=\]
\[\mathrm{Tr}_{q^n/q^t}\left( f\left(\frac{x}{\epsilon^{q^t}-\epsilon} \right) \right)+\mathrm{Tr}_{q^n/q^t}\left( \frac{\epsilon^{q^t} x}{\epsilon^{q^t}-\epsilon}\right),\]
where $\epsilon=\xi^{-1}$.
\end{theorem}
\begin{proof}
Let $U'=\F_{q^t} \times \epsilon S_{f,\xi}$, clearly $L_U$ and $L_{U'}$ are $\mathrm{PGL}(2,q^n)$-equivalent,
$\{1,\epsilon\}$ is an $\F_{q^t}$-basis of $\fqn$ and 
\[ \epsilon S_{f,\xi}=\{\epsilon z+ f(z) \colon z \in \F_{q^t}\}.\]
The linear set $L_{U'}$ is of the form of those in Section \ref{sec:projection} since $ \F_{q^t}\cap \epsilon S_{f,\xi}=\{0\}$, that is $L_{U'}$ is $\mathrm{P}\Gamma\mathrm{L}(2,q^n)$-equivalent to $L_p$ where $p(x)$ can be chosen as the projection map from $\epsilon S_{f,\xi}$ onto $\F_{q^t}$.
So,
\begin{itemize}
    \item $p(\alpha) \in \F_{q^{t}}$ for every $\alpha \in \fqn$;
    \item $p(\alpha)=\alpha$, for every $\alpha \in \F_{q^{t}}$;
    \item $p(\alpha)=0$, for every $\alpha \in \epsilon S_{f,\xi}$.
\end{itemize}
Suppose that $p(x)=\sum_{i=0}^{n-1}a_ix^{q^i}\in \mathcal{L}_{n,q}$.
Let $\alpha \in \fqn$, then $p(\alpha)=p(\alpha)^{q^{t}}$, that is
\[ \sum_{i=0}^{n-1} a_i \alpha^{q^i} = \sum_{j=0}^{n-1} a_j^{q^{t}} \alpha^{q^{j+t}} \]
for every $\alpha \in \fqn$.
This implies that $a_{t+j}=a_j^{q^{t}}$ for each $j \in \{0,\ldots,t-1\}$.
So,
\[ p(x)=\mathrm{Tr}_{q^n/q^{t}}(F(x)), \]
where $F(x)=\sum_{i=0}^{t-1}a_ix^{q^i}$.
Since $p(\alpha)=\alpha$ for each $\alpha \in \F_{q^{t}}$ we have
\[ \alpha=\mathrm{Tr}_{q^n/q^{t}}\left(\sum_{i=0}^{t-1}a_i\alpha^{q^i}\right)= \sum_{i=0}^{t-1}\mathrm{Tr}_{q^n/q^{t}}(a_i)\alpha^{q^i}.\]
Therefore,
\begin{equation}\label{eq:cond2w} \mathrm{Tr}_{q^n/q^{t}}(a_i)=\left\{\begin{array}{ll} 1, & \text{if}\, i=0,\\
0, & \text{otherwise}.\end{array} \right. \end{equation}
As for every $z \in \F_{q^{t}}$ we have $\mathrm{Tr}_{q^n/q^{t}}(F(\epsilon z+f(z)))=0$ and $p(f(z))=\mathrm{Tr}_{q^n/q^t}(F(f(z)))=f(z)$, then
\[ \sum_{i=0}^{t-1}\mathrm{Tr}_{q^n/q^{t}}\left(  a_i \epsilon^{q^i} \right) z^{q^i}+f(z)=0, \]
that is $\mathrm{Tr}_{q^n/q^{t}}(a_i\epsilon^{q^i})=-A_i$ for each $i \in \{0,\ldots,t-1\}$.
Using \eqref{eq:cond2w}, we obtain
\[ \left\{\begin{array}{ll}
a_0=\frac{A_0+\epsilon^{q^{t}}}{\epsilon^{q^{t}}-\epsilon},\\
a_i=\frac{A_i}{\epsilon^{q^{i+t}}-\epsilon^{q^i}},\,\,\text{for every}\,\, i\in \{1,\ldots,t-1\}.
\end{array}\right. \]
This concludes the proof.
\end{proof}

Choosing $f(z)=z^{q^s}$ with $\gcd(s,t)=1$ we obtain the following result as a consequence of Corollary \ref{cor:ex1}.

\begin{corollary}
Let $n=2t$ be an even positive integer and let $s$ be a positive integer such that $\gcd(s,t)=1$. Let $\{1,\xi\}$ be an $\F_{q^{t}}$-basis of $\fqn$ and denote $\epsilon=\xi^{-1}$.
The $\fq$-linear set $L_p$ defined by
\[ p(x)=\mathrm{Tr}_{q^n/q^{t}}\left( \frac{\epsilon^{q^{t}}}{\epsilon^{q^{t}}-\epsilon}x+  \frac{1}{\epsilon^{q^{s+t}}-\epsilon^{q^s}} x^{q^s} \right), \]
has two points of weight $n/2$ and all the other points have weight one.
\end{corollary}

Now we determine a polynomial $p(x)$ such that $L_U$ is $\mathrm{P}\Gamma\mathrm{L}(2,q^n)$-equivalent to $L_p$, where $L_U$ is as in Corollary \ref{cor:ex2}.

\begin{theorem}
Let $s$ be a positive integer coprime with $t$ and $\xi \in \F_{q^{2t}}$ such that $\{1,\xi\}$ is an $\F_{q^t}$-basis of $\F_{q^{2t}}$.
Let $f(x)=x^{q^s}$ and $g(x)=\mu x^{q^s}$ with $\mu \in \F_{q^t}$ such that $\N_{q^t/q}(\mu)\neq 1$ and $\N_{q^{t}/q}(-\xi^{q^t+1}\mu)\neq (-1)^t$.
Let $U=T_{\mu x^{q^s},\xi} \times  S_{x^{q^s},\xi}$.
Then $L_U$ is $\mathrm{P\Gamma L}(2,q^n)$-equivalent $L_p$, where
\[p(x)=\sum_{k=0}^{n-1}\left( \sum_{\ell=0}^{t-1} (u_\ell+u_\ell^{q^s}\xi ){\lambda_{\ell}^*}^{q^k} \right)x^{q^k},\]
$\{u_0,\ldots,u_{t-1}\}$ is an $\fq$-basis of $\F_{q^t}$ and $(\lambda_0^*,\ldots,\lambda_{n-1}^*)$ is the dual basis of $(u_0+\mu u_0^{q^s}\xi,\ldots,u_{t-1}+ \mu u_{t-1}^{q^s}\xi,u_0+u_0^{q^s}\xi,\ldots,u_{t-1}+u_{t-1}^{q^s}\xi)$.
\end{theorem}
\begin{proof}
First we note that $S_{x^{q^s},\xi}\cap T_{\mu x^{q^s},\xi}=\{0\}$.
Indeed, suppose that $w \in S_{x^{q^s},\xi}\cap T_{\mu x^{q^s},\xi}\setminus\{0\}$, then there would exist $u,v \in \F_{q^t}$ such that
\[ w=u+u^{q^s}\xi=v+\mu v^{q^s}\xi, \]
i.e.\ $\mu=1$, a contradiction.
Now, since $\{u_0,\ldots,u_{t-1}\}$ is an $\fq$-basis of $\F_{q^t}$ then  $\{u_0+u_0^{q^s}\xi,\ldots,u_{t-1}+u_{t-1}^{q^s}\xi\}$ is an $\fq$-basis of $S_{x^{q^s},\xi}$ and $\{u_0+\mu u_0^{q^s}\xi,\ldots,u_{t-1}+\mu u_{t-1}^{q^s}\xi\}$ is a $\fq$-basis of $T_{\mu x^{q^s},\xi}$.
We can now apply Theorem \ref{clas} to get the assertion.
\end{proof}

\section{Polynomials defining linear sets of minimum size}\label{sec:minsize}

Recently, in \cite{DVdV}, Jena and Van de Voorde introduced the following family of linear sets having minimum size.

\begin{theorem}\cite[Theorem 2.7]{DVdV}\label{th:constructionVdV}
Let $\lambda\in \fqn \setminus \fq$ be an element generating a degree $s$-extension of $\fq$ and 
\[ L=\{ \langle (\alpha_0+\alpha_1 \lambda+\ldots+\alpha_{t_1-1}\lambda^{t_1-1},\beta_0+\beta_1 \lambda+\ldots+\beta_{t_2-1}\lambda^{t_2-1}) \rangle_{\fqn} \colon \alpha_i,\beta_i \in \fq,\,\]\[\text{not all zero},\, 1 \leq t_1,t_2, t_1+t_2 \leq s+1   \}. \]
Then $L$ is an $\fq$-linear set of $\PG(1,q^n)$ of rank $k=t_1+t_2$ with $q^{k-1}+1$ points.
Let $t_1\leq t_2$, then 
\begin{itemize}
\item the point $\langle (0,1)\rangle_{\fqn}$ has weight $t_2$;
\item there are $q^{t_2-t_1+1}$ points of weight $t_1$ different from $P$;
\item there are $q^{k-2i+1}-q^{k-2i-1}$ points of weight $i \in \{1,\ldots, t_1-1\}$.
\end{itemize}
\end{theorem}

In particular, when $k=n$, the linear set constructed in Theorem \ref{th:constructionVdV} has rank $n$ and $\{1,\lambda,\ldots,\lambda^{n-1}\}$ is an $\fq$-basis of $\fqn$. In such a case, there exists a $q$-polynomial $f(x) \in \mathcal{L}_{n,q}$ such that $L$ is $\mathrm{P}\Gamma\mathrm{L}(2,q^n)$-equivalent to
\[ L_f=\{ \langle (x,f(x))\rangle_{\fqn} \colon x \in \fqn^* \}. \]
The open problem (A) of \cite{DVdV} regards the determination of such polynomials.
Our aim now is to find such $f(x)$'s.

Let $t_2=s$ and $t=t_1=n-s$. Observe that by applying the projectivity $\varphi$ defined by $\left(\begin{array}{cc} 1 & \lambda^t \\ 0 & \lambda^t \end{array}\right)$ to the linear set $L$, we obtain
\[ \varphi(L)=\{ \langle (\alpha_0+\alpha_1 \lambda+\ldots+\alpha_{t-1}\lambda^{t-1}+\beta_0\lambda^t+\beta_1 \lambda^{t+1}+\ldots+\beta_{s-1}\lambda^{n-1},\]\[\beta_0\lambda^t+\beta_1 \lambda^{t+1}+\ldots+\beta_{s-1}\lambda^{n-1}) \rangle_{\fqn} \colon \alpha_i,\beta_i \in \fq,\,\text{not all zero} \}. \]
Hence, $\varphi(L)=L_{p_{T,S}}$ where $T=\langle 1,\lambda,\ldots,\lambda^{t-1}\rangle_{\fq}$, $S=\langle \lambda^t,\ldots,\lambda^{n-1} \rangle_{\fq}$ and $p_{T,S}$ is the projection map from $T$ onto $S$. 
Therefore, we can now apply Theorem \ref{clas} to such linear sets.

\begin{theorem}\label{th:polminsize}
Let $\lambda\in \fqn \setminus \fq$ such that $\fq(\lambda)=\fqn$ and let $f(x)=a_0+a_1x+\ldots+a_{n-1}x^{n-1}+x^n\in \fq[x]$ be the minimal polynomial of $\lambda$ over $\fq$ and $\delta=f'(\lambda)$. 
Let 
\[ L=\{ \langle (\alpha_0+\alpha_1 \lambda+\ldots+\alpha_{t-1}\lambda^{t-1},\beta_0+\beta_1 \lambda+\ldots+\beta_{n-t-1}\lambda^{n-t-1}) \rangle_{\fqn} \colon \alpha_i,\beta_i \in \fq,\,\]\[\text{not all zero}  \}. \] 
Then $L$ is $\mathrm{PG}\mathrm{L}(2,q^n)$-equivalent to 
\[ L_p=\left\{\left\langle \left(x,p(x)\right)\right\rangle_{\fqn} \colon x \in \fqn \right\}, \]
where 
\[ p(x)=\sum_{i=t}^{n-1}\lambda^i\mathrm{Tr}_{q^n/q}((\lambda^i)^* x)=\sum_{j=0}^{n-1}A_jx^{q^j} \]
and
\[A_j=\frac{1}{\delta^{q^j}}\sum_{i=t}^{n-1}\sum_{h=1}^{n-i} \lambda^{i+(h-1)q^j} a_{i+h}^{q^j}.\]
\end{theorem}
\begin{proof}
As already seen, $L$ is $\mathrm{PGL}(2,q^n)$-equivalent to the linear set $L_{p_{T,S}}$, where $T=\langle 1,\lambda,\ldots,\lambda^{t-1}\rangle_{\fq}$, $S=\langle \lambda^t,\ldots,\lambda^{n-1} \rangle_{\fq}$ and $p_{T,S}$ is the projection map from $T$ onto $S$. 
Hence the ordered $\fq$-basis $\mathcal{B}=(1,\lambda,\ldots,\lambda^{n-1})$ is a polynomial basis of $\fqn$, and so $\fq(\lambda)=\fqn$.
Let $g(x)=\frac{f(x)}{x-\lambda}=\gamma_0+\gamma_1x+\gamma_{n-1}x^{n-1}\in\fqn[x]$ and $\delta=f'(\lambda)$. By Corollary \ref{cor:dualbasis}, the dual basis of $\mathcal{B}$ is 
\[\mathcal{B}^*=(\delta^{-1}\gamma_0,\ldots,\delta^{-1}\gamma_{n-1}),\]
with 
\[ \gamma_i=\sum_{j=1}^{n-i} \lambda^{j-1}a_{i+j}, \]
for every $i \in \{0,\ldots, n-1\}$.
By Theorem \ref{clas}, we have
\[ p(x)=p_{T,S}(x)=\sum_{i=t}^{n-1}\lambda^i\mathrm{Tr}_{q^n/q}((\lambda^i)^* x)=\sum_{j=0}^{n-1} A_j x^{q^j}, \]
where
\[ A_j=\sum_{i=t}^{n-1}\frac{\lambda^i}{\delta^{q^j}} \gamma_i^{q^j} =\sum_{i=t}^{n-1}\frac{\lambda^i}{\delta^{q^j}} \left(\sum_{h=1}^{n-i} \lambda^{h-1}a_{i+h}\right)^{q^j} =\frac{1}{\delta^{q^j}}\sum_{i=t}^{n-1}\sum_{h=1}^{n-i} \lambda^{i+(h-1)q^j} a_{i+h}^{q^j}. \]
\end{proof}

When the fixed basis $\mathcal{B}$ is weakly self dual the coefficients of $p(x)$ can be reduced to a simpler form.

\begin{corollary}\label{cor:polminsizespecial}
Let $\lambda\in \fqn \setminus \fq$ such that $\fq(\lambda)=\fqn$.
\begin{itemize}
    \item If $f(x)=x^n-d\in \fq[x]$ is the minimal polynomial of $\lambda$ over $\fq$, then the coefficients of the polynomial $p(x)$ of Theorem \ref{th:polminsize} are
    \[A_j=\frac{1}n \sum_{i=t}^{n-1}\lambda^{i(1-q^j)}.\]
    \item If $f(x)=x^n-cx-1\in \fq[x]$ is the minimal polynomial of $\lambda$ over $\fq$, then the coefficients of the polynomial $p(x)$ of Theorem \ref{th:polminsize} are
    \[ A_j=\frac{-c+\lambda^{q^j(n-k)}}{n\lambda^{q^j(n-1)}-ck\lambda^{q^j(k-1)}} \left( \sum_{i=t}^{k-1} \lambda^{i+q^j(k-i-1)} +\sum_{i=k}^{n-1} \lambda^{i+q^j(n+k-i-1)} \right) \]
    if $t<k$, and 
    \[ A_j=\frac{-c+\lambda^{q^j(n-k)}}{n\lambda^{q^j(n-1)}-ck\lambda^{q^j(k-1)}}  \left( 
    \sum_{i=t}^{n-1} \lambda^{i+q^j(n+k-i-1)} \right) \]
    if $t\geq k$.
\end{itemize}
\end{corollary}
\begin{proof}
First suppose that $f(x)=x^n-d$. By Corollary \ref{cor:binetrin}, the dual basis of $\mathcal{B}=(1,\lambda,\ldots,\lambda^{n-1})$ is given by
 \[ \mathcal{B}^*=(\delta^{-1}\lambda^{n-1},\ldots,\delta^{-1}\lambda,\delta^{-1}), \]
with $\delta=n\lambda^{n-1}$.
By applying Theorem \ref{clas}, then $\displaystyle p(x)=\sum_{i=t}^{n-1}\lambda^i\mathrm{Tr}_{q^n/q}((\lambda^i)^*x)=\sum_{i=t}^{n-1}\lambda^i\mathrm{Tr}_{q^n/q}(\delta^{-1} \lambda^{n-i-1} x)$ and the assertion follows.
Now, assume that $f(x)=x^n-cx-1$. By Corollary \ref{cor:binetrin}, the dual basis of $\mathcal{B}=(1,\lambda,\ldots,\lambda^{n-1})$ is given by
\[ \mathcal{B}^*=(\delta^{-1}\lambda^{k-1},\delta^{-1}\lambda^{k-2},\ldots,\delta^{-1},\delta^{-1}\lambda^{n-1},\ldots,\delta^{-1}\lambda^k), \]
where $\delta=\frac{n\lambda^{n-1}-ck\lambda^{k-1}}{-c+\lambda^{n-k}}$.
By applying Theorem \ref{clas}, then $\displaystyle p(x)=\sum_{i=t}^{n-1}\lambda^i\mathrm{Tr}_{q^n/q}((\lambda^i)^*x)$. If $t < k$ then
\[ p(x)=\sum_{i=t}^{k-1}\lambda^i\mathrm{Tr}_{q^n/q}((\lambda^i)^*x)+\sum_{i=k}^{n-1}\lambda^i\mathrm{Tr}_{q^n/q}((\lambda^i)^*x)= \]
\[ \sum_{i=t}^{k-1}\lambda^i\mathrm{Tr}_{q^n/q}(\delta^{-1}\lambda^{k-i-1}x)+\sum_{i=k}^{n-1}\lambda^i\mathrm{Tr}_{q^n/q}(\delta^{-1}\lambda^{n+k-i-1}x), \]
and if $t\geq k$ then
\[ p(x)=\sum_{i=t}^{n-1}\lambda^i\mathrm{Tr}_{q^n/q}(\delta^{-1}\lambda^{n+k-i-1}x). \]
So the assertion holds.
\end{proof}

\section*{Acknowledgements}

The research was supported by the project ``VALERE: Vanvitelli pEr la RicErca" of the University of Campania ``Luigi Vanvitelli'' and was partially supported by the Italian National Group for Algebraic and Geometric Structures and their Applications (GNSAGA - INdAM).

\begin{small}

\end{small}

\newpage

\appendix
\section{Polynomial shape of linear sets in $\mathrm{PG}(1,q^4)$}\label{sec:pg1q4}

In \cite{BoPol} $\fq$-linear blocking sets of $\mathrm{PG}(2,q^4)$ were classified; see also \cite{PolitoPolverino}.
The authors also showed that $\fq$-linear blocking sets in $\mathrm{PG}(2,q^n)$ are \emph{simple}, that is if $L_U$ and $L_W$ are two $\mathrm{P}\Gamma\mathrm{L}(3,q^n)$-equivalent $\fq$-linear blocking sets in $\mathrm{PG}(2,q^n)$ if and only $U$ and $W$ are $\Gamma\mathrm{L}(3,q^n)$-equivalent.
So that, from the above mentioned classification follows the classification of $\fq$-linear set of rank $4$ in $\mathrm{PG}(1,q^4)$; see also \cite[Corollary 5.4]{BMZZ}.

In \cite[5. Section 6]{CsMPq5} and in \cite[Open Problem 8.4]{BMZZ} was stated the problem of determining the polynomial shape of the $\fq$-linear sets of rank $4$ of $\mathrm{PG}(1,q^4)$.
This section is devoted to explicitly solve this problem.
Using the notation of \cite{BoPol} we ony need to determine the polynomial shape of $L_{C_{12}}$, $L_{C_{15}}$, $L_{B_{22}}$ and  $L_{C_{13}}$.

By \cite[Proposition 2.9]{DVdV}, $L_{C_{12}}$ corresponds to the linear sets described in Section \ref{sec:minsize}, so that its polynomial shape is a direct consequence of Theorem \ref{th:polminsize}.


Regarding the linear set $L_{B_{22}}$, it is defined by a subspace $U$ of the form $U=T\times S$,
where $T=\F_{q^2}$ and $S$ is a scattered subspace in $\F_{q^4}$ w.r.t.\ $\F_{q^2}$ (that is, $\dim_{\fq}(S\cap \langle u \rangle_{\F_{q^2}})\leq 1$, for every $u \in \F_{q^4}$). 
Hence, $L_{B_{22}}$ is $\mathrm{P}\Gamma\mathrm{L}(2,q^4)$-equivalent to the linear set introduced in Corollary \ref{cor:ex1} and its polynomial shape can be determined using Theorem \ref{th:pol2w}.
For the remaining cases we can use the above machinery, obtaining the following results.




\begin{proposition}
Let $L_{C_{13}}$ be as in \cite{BoPol}. Let $\lambda \in \F_{q^4}$ be such that $\mathcal{B}=(1,\eta_1,\lambda,\lambda \eta_2)$ is an $\F_q$-basis of $\F_{q^4}$. Let $\mathcal{B}^*=(\xi_0^*,\xi_1^*,\xi_2^*,\xi_{3}^*)$ be the dual basis of $\mathcal{B}$. Then $L_{C_{13}}$ is $\mathrm{P}\Gamma\mathrm{L}(2,q^4)$-equivalent to $L_p=\{\langle (x,p(x))\rangle_{\F_{q^n}}:x \in \fqn^*\setminus \{\boldsymbol{0} \}\},$ where 
\[
p(x)=\sum_{j=0}^{3}\left( \xi_2^{*q^j} + \eta_2 \xi_3^{*q^j} \right)x^{q^j}.
\]
\end{proposition}



\begin{proposition}
Let $L_{C_{15}}$ be as in \cite{BoPol}. Let $\lambda \in \F_{q^4}$ be such that $\mathcal{B}=(-\eta_1,\lambda, 1, -\lambda \eta_1-\eta_2)$ is an $\F_q$-basis of $\F_{q^4}$. Let $\mathcal{B}^*=(\xi_0^*,\xi_1^*,\xi_2^*,\xi_{3}^*)$ be the dual basis of $\mathcal{B}$. Then $L_{C_{15}}$ is $\mathrm{P}\Gamma\mathrm{L}(2,q^4)$-equivalent to $L_p=\{\langle (x,p(x))\rangle_{\F_{q^n}}:x \in \fqn^*\setminus \{\boldsymbol{0} \}\},$ where 
\[
p(x)=\mathrm{Tr}_{q^n/q}(\xi_1^*x)+ \eta_1 \mathrm{Tr}_{q^n/q}(\xi_3^* x).
\]
\end{proposition}



By summing up, we obtain the following result.

\begin{theorem}
Let $L_U$ be an $\fq$-linear set of rank $4$ in $\PG(1,q^4)$. 
\begin{itemize}
    \item[(o)] If $|L_U|=q^2+1$, then $L_U$ is a Baer subline of ${\rm PG}(1,q^4)$ and all points of $L_U$ have weight $2$.
    \item [(i)] If $|L_U|=q^3+q^2+q+1$, then all the points have weight $1$ and $L_U$ is $\mathrm{P}\Gamma\mathrm{L}(2,q^4)$-equivalent to $L_{x^q+\delta x^{q^3}}$ for some $\delta \in \mathbb{F}_{q^4}$ with $\N_{q^4/q}(\delta)\ne 1$.
    \item [(ii)] If $|L_U|=q^3+1$, then $L_U$ is $\mathrm{P}\Gamma\mathrm{L}(2,q^4)$-equivalent either to $L_{\mathrm{Tr}_{q^4/q}(x)}$ 
    and $L_U$ has $1$ point of weight $3$ and $q^3$ points of weight $1$;
    or to
    $L_p$ where $p(x)=A_0 x+A_1x^q+A_2x^{q^2}+A_3x^{q^3}$,
    \[ A_j=\lambda^2 a_3^{q^j}+\lambda^{2+q^j}a_0^{q^j}+\lambda^3 a_0^{q^j}, \]
    $j \in \{0,1,2,3\}$, $\{1,\lambda,\lambda^2,\lambda^3\}$ is an $\fq$-basis of $\F_{q^4}$, $f(x)=a_0+a_1x+a_2x^2+a_{3}x^{3}+x^4\in \fq[x]$ is the minimal polynomial of $\lambda$ over $\fq$ and $\delta=f'(\lambda)$, and $L_U$ has $q+1$ points of weight $2$ and $q^3-q$ points of weight $1$.
    \item [(iii)] If $|L_U|=q^3+q^2+1$, then $L_U$ has $1$ point of weight $2$ and $q^3+q^2$ points of weight $1$ and $L_U$ is $\mathrm{P}\Gamma\mathrm{L}(2,q^4)$-equivalent either to $L_{x^q-x^{q^3}}$ or to $L_p$ where 
    \[
    p(x)=\mathrm{Tr}_{q^n/q}(\xi^*x)+ \eta_1 \mathrm{Tr}_{q^n/q}(\xi_3^* x),
    \]
    $(-\eta_1,\lambda, 1, -\lambda \eta_1-\eta_2)$ is an ordered $\F_q$-basis of $\F_{q^4}$, $\eta_1 \notin \F_{q^2}$ and $(\xi_0^*,\xi_1^*,\xi_2^*,\xi_{3}^*)$ is its dual basis.
    \item [(iv)] If $|L_U|=q^3+q^2-q+1$, then  $L_U$ has $2$ points of weight $2$ and $q^3+q^2-q-1$ points of weight $1$ and $L_U$ is $\mathrm{P}\Gamma\mathrm{L}(2,q^4)$-equivalent either to $L_p$ where 
    \[ p(x)=\frac{A_0+\overline{\xi}^{q^2}}{\overline{\xi}^{q^2}-\overline{\xi}}x+\frac{A_1}{\overline{\xi}^{q^3}-\overline{\xi}}x^q+\frac{A_0+\overline{\xi}}{\overline{\xi}-\overline{\xi}^{q^2}}x^{q^2}+\frac{A_1^{q^2}}{\overline{\xi}^{q}-\overline{\xi}^{q^2}}x^{q^3}, \]
    $\{1,\overline{\xi}\}$ is an $\F_{q^2}$-basis of $\F_{q^4}$ and with $A_0,A_1 \in \F_{q^2}$ and $A_1\ne 0$, or to
    $L_p$ with
    \[
    p(x)=\sum_{j=0}^{3}\left( \xi_2^{*q^j} + \eta_2 \xi_3^{*q^j} \right)x^{q^j},
    \]
    $(1,\eta_1,\lambda,\lambda \eta_2)$ is an ordered $\F_q$-basis of $\F_{q^4}$ having $\eta_1,\eta_2 \in \F_{q^4}\setminus\F_{q^2}$ and $(\xi_0^*,\xi_1^*,\xi_2^*,\xi_{3}^*)$ is its dual basis of $\mathcal{B}$. 
\end{itemize}
\end{theorem}

\end{document}